\documentclass[11pt]{amsart}
\usepackage{amsmath,amssymb,amsthm,tikz,amsfonts,graphicx,tikz,pgfplots}

\usepackage{mdwlist}
\usepackage{todonotes}
\usepackage{enumerate}
\usepackage[utf8]{inputenc}
\RequirePackage{color}
\RequirePackage[colorlinks, urlcolor=my-blue,linkcolor=my-red,citecolor=my-green]{hyperref}
\usepackage{cite}
\definecolor{my-blue}{rgb}{0.0,0.0,0.6}
\definecolor{my-red}{rgb}{0.5,0.0,0.0}
\definecolor{my-green}{rgb}{0.0,0.5,0.0}
\definecolor{nicos-red}{rgb}{0.75,0.0,0.0}
\definecolor{light-gray}{gray}{0.6}
\definecolor{really-light-gray}{gray}{0.8}
\definecolor{sussexg}{rgb}{0.0,0.5,0.5}
\definecolor{sussexp}{rgb}{0.5,0.0,0.5}

\usepackage{placeins}

\let\Oldsection\section
\renewcommand{\section}{\FloatBarrier\Oldsection}

\let\Oldsubsection\subsection
\renewcommand{\subsection}{\FloatBarrier\Oldsubsection}

\let\Oldsubsubsection\subsubsection
\renewcommand{\subsubsection}{\FloatBarrier\Oldsubsubsection}

\newcommand{\as}[4]{ L_{#1,#2}^{(#3,#4)} }
\newcommand{\asz}[3]{ L_{#1,#2}^{(#3)} }
 \newcommand{\cgs}[2]{ T_{#1,#2} }
 \newcommand{\is}[4]{ G_{#1,#2}^{(#3,#4)} }
 \newcommand{\isz}[3]{ G_{#1,#2}^{(#3)} }
 \newcommand{\bsz}[2]{ G_{#1,#2}^{(0)} }

 \DeclareMathOperator{\al}{al}
  \DeclareMathOperator{\ind}{ind}

\newcommand{\reg}[2]{R_{#1,#2}}
\newcommand{\regmn}{\reg{m}{n}}
\newcommand{\alreg}[2]{R_{#1,#2}^{(\al)}}
\newcommand{\alregmn}{\alreg{m}{n}}
\newcommand{\indreg}[2]{R_{#1,#2}^{(\ind)}}

\usepackage{nicefrac}

\usepackage[top=1.2in, bottom=1.2in, left=0.8in, right=0.8in]{geometry}

\newcommand{\pimax}{\Gamma^{(\alpha,\beta)}_{0,(m,n)}}
\newcommand{\pim}[2]{\Gamma^{(#1,#2)}_{0,(m,n)}}
\newcommand{\pimo}[1]{\Gamma^{(#1)}_{0,(m,n)}}

\newtheorem{theorem}{\sc Theorem}[section]
\newtheorem{lemma}[theorem]{\sc Lemma}

\newtheorem{corollary}[theorem]{\sc Corollary}

\newtheorem{definition}[theorem]{\it Definition}
\numberwithin{equation}{section}
\theoremstyle{remark}
\newtheorem{remark}[theorem]{Remark}

\newcommand{\be}{\begin{equation}}
\newcommand{\ee}{\end{equation}}

\providecommand{\abs}[1]{\vert#1\vert}

\providecommand{\P}[1]{\langle#1\rangle}
\newcommand{\fl}[1]{\left\lfloor{#1}\right\rfloor}

\newcommand{\env}{\omega}
\newcommand{\w}{\eta}
\newcommand{\wx}{\w^{\textrm{x}}}
\newcommand{\wy}{\w^{\textrm{y}}}
 \newcommand{\abc}{\mathcal{A}}

\newcommand{\rb}[1]{\left(#1\right)}
\newcommand{\ab}[1]{\left[#1\right]}
\newcommand{\set}[1]{\left\{#1\right\}}

\def\bE{\mathbb{E}}
\def\bN{\mathbb{N}}
\def\bP{\mathbb{P}}

\def\bR{\mathbb{R}}
\def\bZ{\mathbb{Z}}

\def\mC{\mathcal{C}}

\def\mR{\mathcal{R}}

\def\e{\varepsilon}

\def\om{\omega}

 \def\Z{\bZ}

\def\R{\bR}
\def\N{\bN}

\def\E{\bE}
\def\P{\bP} 

\definecolor{darkgreen}{rgb}{0.0,0.5,0.0}
\definecolor{darkblue}{rgb}{0.0,0.0,0.3}
\definecolor{nicosred}{rgb}{0.65,0.1,0.1}
\definecolor{light-gray}{gray}{0.7}
\allowdisplaybreaks[1]

\begin{document}
\title[Optimality regions and fluctuations near the soft edge]
{Optimality regions and fluctuations for Bernoulli last passage models}


\author[N.~Georgiou]{Nicos Georgiou}
\address{Nicos Georgiou\\ University of Sussex\\ Department of  Mathematics \\ Falmer Campus\\ Brighton BN1 9QH\\ UK.}
\email{n.georgiou@sussex.ac.uk}
\urladdr{http://www.sussex.ac.uk/profiles/329373} 
\thanks{NG was partially supported by the University of Sussex Strategic development Fund (SDF) and by the EPSRC First grant EP/P021409/1: The flat edge in last passage percolation.  JO was partially supported by an ISM-CRM fellowship}

\author[J.~Ortmann]{Janosch Ortmann}
\address{Janosch Ortmann\\ Université du Québec à Montréal \\ Département de management et technologie \\ Canada.}
\email{ortmann.janosch@uqam.ca}

\urladdr{http://crm.umontreal.ca/~ortmann/}

\keywords{Soft edge, edge results, optimality regions, sequence alignment, discrete Hammersley process, longest common subsequence, Bernoulli increasing paths, Tracy-Widom distribution, last passage time, corner growth models, flat edge.}
\subjclass[2000]{60K35}
\date{\today}
\begin{abstract}
	We study the sequence alignment problem and its independent version, the discrete Hammersley process with an exploration penalty. We obtain rigorous upper bounds for the number of optimality regions 
	in both models near the soft edge. 
	At zero penalty the independent model becomes an exactly solvable model and we identify cases for which the law of the last passage time converges to a Tracy-Widom law. 
\end{abstract}

\maketitle

\section{Introduction}
	\label{sec:intro}
	
	\subsection{Directed growth models}
	
	In this article we study a generalisation of two specific models of 
	directed last passage percolation, namely the \textit{longest common subsequence model} concerning the size of the longest common subsequence between words drawn uniformly from a finite alphabet \cite{Chv-Sank}, 
	 and an independent version introduced in \cite{Sepp-97} as an exactly solvable discrete analogue of the Hammersley process \cite{Ham}. 
	 We call the latter the \textit{independent model}.
	 
	 We study these models near  directions for which the corresponding shape function starts developing a flat segment, which is called the \emph{soft edge} of the model. 
	Both models fit in the general framework \cite{GRS13}, namely there is: 
	\begin{enumerate}[(i)]
	\item The random environment $\env\in\bR^{\Z^2}$, whose law we denote by $\bP$. Each marginal $\env_u$ should be viewed as a random weight placed on site $u \in \Z^2$.
	\item 
	A collection $\Pi$ of admissible paths on $\Z^2$. A path $\pi$ from $u$ to $v$ is uniquely identified by an ordered sequence of integer sites, 
	so when necessary we write $\pi = \{ u = u_0, u_1, \ldots, u_\ell = v\}$. A path $\pi$ is admissible
	if and only if its increments $z_k=u_{k}-u_{k-1}$ are contained in a finite set $\mR\subset\Z^2$. 
	For $u,v\in\Z^2$ we denote the set of admissible paths from $u$ to $v$ by $\Pi_{u,v}$. 
	It is a requirement that $\P$ is stationary and ergodic under shifts 
	$T_{z}, z \in \mathcal R$.	
	\item A measurable potential function $V: \bR^{\Z^2} \times \mathcal R^\ell \to \bR$. 
	For the two models under investigation we always have $\ell=1$ and $V$ is a bounded function, thus satisfying the technical assumptions of \cite{GRS13}.
	\end{enumerate}

	The \emph{point-to-point last passage time} from $u$ to $v$ is the random variable $G^{V}$ defined by 
	\be
	G^{V}_{u,v} = \max_{\pi \in \Pi_{u,v}}\Big\{ \sum_{u_k \in \pi } V\rb{T_{u_{k}}\env, z_{k+1}} \Big\}. 
	\ee

	A well studied version of the model is the \emph{corner growth model}, for which 
	$\mathcal{R} = \{ e_1, e_2\}$, the coordinates of $\env$ are i.i.d.~under $\bP$ and 
	the potential $V$ for the corner growth model is defined by 
	\be\label{eq:potcgm}
	V(\om, z) = \om_0, \quad z \in \mathcal R = \{ e_1, e_2\}. 
	\ee
	Whenever we are referring to last passage time under this potential and these admissible steps, we will use $T$ instead of $G^V$. It is expected that under some regularity assumptions on the moments and continuity of $\env_0$,  the asymptotic behaviour of $T$ 
	(e.g.\ fluctuation exponents for $T$ and the maximal path,  distributional limits, etc) is environment-independent. 
	This is suggested by results available for the two much-studied exactly solvable models when $\env_0$ is exponentially or geometrically distributed and further evidenced by the general theory in
	\cite{GRS13,cocycle, geodesics} and the edge results of \cite{Bodineau-Martin,Martin04}, as we discuss later.

	The main models in this article have set of admissible steps  $\mathcal R = \{e_1, e_2, e_1 + e_2\}$
	and the  coordinates of the environment take values in $\set{0,1}$. Our choice of potential is a two-parameter family of bounded functions, indexed by two non-negative parameters $\alpha$ and $\beta$:
	 \be\label{eq:vab}
	 V_{\alpha, \beta}(\env, z)  = \begin{cases}
	 	\env_0-\alpha\rb{1-\env_0} \quad & \text{if }z=e_1+e_2\\
		-\beta & \text{if } z\in\set{e_1,e_2}.
	 \end{cases}
	\ee
	This particular choice of potential is inspired by a problem which appears in computational molecular biology, computer science and algebraic statistics, as we explain at the end of this introduction. 
%
Our strongest results are obtained when $\alpha = \beta = 0$ and the marginals of $\env$ are i.i.d.~Bernoulli random variables on $\set{0,1}$ with parameter $p\in (0,1)$, because we then obtain a solvable model \cite{Priezzev-Schutz}. This will be referred to as the \emph{independent model}, and the passage time from $(0,0)$ to $(m,n)$ is denoted $\is mn\alpha\beta$ when both $\alpha$ and $\beta$ are important. When $\alpha=0$ we further simplify notation by $\isz mn\beta= \is mn0\beta$.  The special case $\alpha=\beta=0$ was studied in \cite{Geo-2010,Bas-Enr-Ger-Gou, Sepp-97}. Asymptotic results as $p$ tends to zero were obtained in \cite{Kiw-Loe-Mat}.
	
	We consider a rectangle of height $n$ and width $m_n= \nicefrac np - x n^a$ for $a\in (0,1)$ and show that the fluctuations of $\isz{m_n}n0$ converge, suitably rescaled, to the Tracy--Widom GUE distribution. The size of the rectangle is not arbitrary. A justification for this option comes by looking at the limiting shape function 
	\[
	g_{pp}(t) = \lim_{n\to \infty} \frac{\isz{\fl{nt}}n0}{n},
	\]
continuous in $t$. When $t > \nicefrac1p$ the function has a \emph{flat edge}: $g_{pp}(t) = 1$.When $p< t < 1/p$,  $g_{pp}(t)$ is strictly concave and when $t < p$, $g_{pp}$ has another flat edge, namely $g_{pp}(t)=t$. Fluctuations of $\isz{\fl{nt}}n0$ are of order $n^{1/3}$ when $t\in \rb{p,\nicefrac1p}$, so by looking at the rectangle $m_n \times n$ we study these fluctuations at the onset of the flat edge, but when macroscopically we converge to the critical point $t =1/p$.

\subsection{Edge results}
There is a coupling of $\bsz{\frac{n}{p} - xn^a}{n}$ with $\cgs{n}{n^{2a-1}}$, which we describe in Section \ref{sec:ORBLIP}.
	This mapping was exploited in \cite{Geo-2010} to obtain the local weak law of large numbers
		\be\label{eq:geoagain}
			\lim_{n\to \infty}\bP\Big\{\Big|\frac{n - \bsz{p^{-1}n - xn^a}{n}}{n^{2a-1}} - \frac{(px)^2}{4(1-p)}\Big| <\e \Big\} = 1
		\ee
		for all $a \in \rb{\nicefrac12,1}$.
		We use the same coupling to obtain a distributional limit for the edge. The coupling classifies results for $\bsz{p^{-1}n - xn^a}{n}$ as ``edge results". The terminology ``edge results'' is motivated by the fact that the last passage time $T$ 
	is studied in a thin rectangle, either with dimensions $ n \times yn$ and letting $y \to 0$ after sending $n \to \infty$ \cite{Martin04}, or with only one macroscopic edge, namely of dimensions  $n \times xn^\gamma$ with $\gamma < 1$.

	 Several results near the edge are universal, in the sense that they do not depend on the particular distribution of the environment. In the sequence we denote the environment for the corner growth model by $\zeta = \{\zeta_u\}_{u \in \Z^2_+}$.
An approximation of i.i.d.\ sums with a Brownian motion \cite{Major-Komlos-Tusnady} was used in \cite{GlynnWhitt91} to obtain the 
	weak law of large numbers, 
	\[
	\frac{\cgs{n}{xn^\gamma} - n \E( \zeta_0)}{\sqrt{\text{Var}(\zeta_0) n^{1+\gamma}}} \Longrightarrow  c \sqrt{x}, \quad (n \to \infty),
	\]
	and simulations lead to the conjecture that $c = 2$. The conjecture was proved in \cite{Sepp-97b} via a coupling with an exclusion process and later in \cite{Baryshnikov} using a random matrix approach. A coupling with the Brownian last passage percolation model \cite{Baryshnikov, OConnellYor01} allow 
	\cite{Bodineau-Martin} to obtain
	\be
	\label{eq:B-M-TW}
	\frac{\cgs{n}{n^\gamma} - n \E( \zeta_0) - \sqrt{\text{Var}(\zeta_0) n^{1+\gamma}}}{n^{\frac{1}{2}-\frac{\gamma}{6}}\sqrt{\text{Var}(\zeta_0)}} \Longrightarrow W, \quad (n \to \infty),
	\ee
	where $W$ is has the \emph{Tracy-Widom GUE distribution} \cite{TracyWidom93}: the limiting distribution of the largest eigenvalue of a GUE random matrix.
	If $\zeta_0$ has exponential moments, \eqref{eq:B-M-TW} holds for all $a \in (0, \nicefrac37)$. 

\subsection{The alignment model} 
	The problem of \emph{sequence alignment} \cite{Nee-Wun, Smi-Wat} can be cast in this framework. Consider two words $\wx=\wx_1\ldots \wx_m$ and $\wy=\wy_1\ldots \wy_n$ formed from a finite alphabet $\abc$. We consider the case where each letter of $\wx$ and $\wy$ is chosen independently and uniformly at random from $\abc$. We are looking for a sequence of elementary operations of minimal cost that transform $\wx$ to $\wy$. These operations are:
	\begin{enumerate}
		\item replace one letter of $\wx$ by another, at a cost $\alpha$
		\item delete a letter  of $\wx$ or insert another letter, each at a cost of $\beta$.
	\end{enumerate}
	Assign a score of 1 for each match and subtract the costs for replacements, deletions and insertions. Each sequence of operations taking $\wx$ to $\wy$ is thus assigned a \emph{score} $\as mn\alpha\beta$, also often called the \emph{objective function}. We will also write $\asz mn\beta$ for $\as mn0\beta$.
	
	 A problem arising in molecular biology \cite{handbook_compbio, pach-sturm,vingron-waterman, Hen-Hen, Xia, Ng-Hen} is to maximise this alignment score. In that context the words $\wx$ and $\wy$ can be DNA strands (with $\abc=\set{A,C,G,T}$), RNA strands ($\abc=\set{A,C,G,U}$) or proteins (with $\abc$ the set of amino acids that make up a protein), and the elementary operations correspond to mutations. A choice of the parameters $\alpha$ and $\beta$ corresponds to a judgement on how frequently each type of mutation occurs. The optimal score for an alignment of $\wx$ with $\wy$ can then be considered a measure of similarity between these words.
	The question also appears in algebraic statistics \cite{PachterSturmfels}: there the objective function is the tropicalisation of a co-ordinate polynomial of a particular hidden Markov model.
	
	The special case $\alpha=\beta=0$ corresponds to the problem of finding longest common subsequence (LCS) of the words $\wx$ and $\wy$, which has been intensively studied by computer scientists \cite{Hirschberg75, Maier78, Ber-Hak-Rai,Mas-Pat} and mathematicians \cite{Chv-Sank, Lem-Mat-Vol,Lem-Mat,Hou-Mat-16,Gon-Hou-Lem,Ams-Mat-Vac}.
	
	On the other hand, the alignment score $\as mn\alpha\beta$ is the last passage time \eqref{eq:vab} in environment 
	\begin{align}
		\env_{ij} & = \begin{cases}
			1 \quad&\text{if } \wx_i=\wy_j\\
			0 & \text{otherwise}, \label{eq:alenv}
		\end{cases}
	\end{align}
	i.e. the marginals of $\env$ are (correlated) Bernoulli random variables with parameter ${\abs\abc}^{-1}$. The model with this choice of environment is referred to as the \emph{alignment model}.
		
\begin{figure}	\label{fig:al1}
		\begin{center}
		\begin{tikzpicture}[>=latex,scale=0.7]
			\draw[<->](6.5,0)--(0,0)--(0, 6.5); 
			\draw[line width=3pt, color=sussexg](0,0)--(1,0)--(2,1)--(3,2)--(4,3)--(4,4)--(6.3,6.3);
			\foreach \x in {1,2,...,6}
							{
							\draw(\x,0)--(\x,6.3);
							\draw(0,\x)--(6.3, \x);
							\fill[color=white](\x,0)circle(1.3mm);
							\draw(\x,0)circle(1.3mm);
							\fill[color=white](0,\x)circle(1.3mm);
							\draw(0,\x)circle(1.3mm);
							}				
			\foreach \x in {1,2,...,6}
						{
						\foreach \y in {1,2,...,6}	
							{
							\fill[color=white](\x,\y)circle(1.3mm);
							\draw(\x,\y)circle(1.3mm);
							}
						}
			\foreach \x in {1,2, 4,6}
						{
						\foreach \y in {1,3,4,6}	
							{
							\fill[color=nicos-red](\x,\y)circle(1.3mm);
							\draw(\x,\y)circle(1.3mm);
							\draw(\x,-0.5)node{$A$};
							\draw(-0.5,\y)node{$A$};
							}
						}
			\foreach \x in {3,5}
						{
						\foreach \y in {2,5}	
							{
							\fill[color=nicos-red](\x,\y)circle(1.3mm);
							\draw(\x,\y)circle(1.3mm);
							\draw(\x,-0.5)node{$B$};
							\draw(-0.5,\y)node{$B$};
							}
						}	
			\draw[line width=2pt, color=nicos-red](4,4)circle(4mm);							\end{tikzpicture}
		\end{center}
	\caption{Environment generated by the two strings $AABABA$ and $ABAABA$. Colored dots correspond to the value 1, white dots to the value 0. The thickset path is a maximal path in this environment, from $(0,0)$ to $(6,6)$ with minimal number of vertical or horizontal steps (just 2 in this case). 
	When $\alpha=0$, the illustrated path has score $5 - 2 \beta$ since the environment only contributes to the weights if collected by a diagonal step. The score coincides with the last passage time for $\beta \le 1/2$. For $\alpha=0$ and $\beta > 1/2$ the main diagonal is optimal, with score equal to $4$. These are the only two optimal paths, so there are two optimality regions.}
\end{figure}
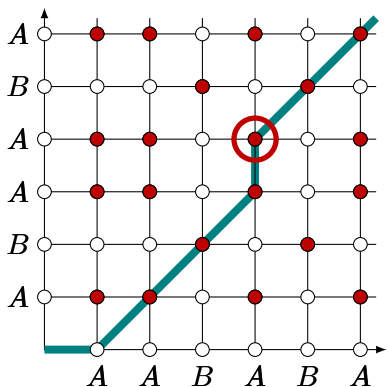
	
A deletion of a character in $\wx$ corresponds to a horizontal step ($e_1$) in the last passage model, whereas an insertion of a letter into $\wx$ corresponds to a vertical step ($e_2$). Replacing a letter in $\wx$ by another corresponds to a diagonal step ($e_1+e_2$) onto a point $(i,j)$ where $\env_{ij}=0$, whereas any letter left alone (i.e. a successful alignment) corresponds to a diagonal step onto a point $(i,j)$ where $\env_{ij}=1$. The path in Figure \ref{fig:al1} corresponds to the alignment
	\begin{align*}
		&\wx: A\, A\, B\,A\, - \!B \, A \\
		&\wy: - \, A \, B\, A\, A\, B\, A 
	\end{align*}
	in which the bar under the first $A$ of $\wx$ corresponds to deleting the letter $A$ from $\wx$ while the bar in $\wx$ corresponds to inserting the letter $A$ there. 
	A convenient way to look at this is that the bars, called \emph{gaps}, are used to stretch the two words appropriately so that different matchings are obtained.	
	
\subsection{Optimality regions.} 
Which paths are optimal depends on the choice of parameters $\alpha,\beta$. In molecular biology these parameters are often chosen ad hoc and it is not clear that there is a single `right' choice \cite{vingron-waterman}. An alternative approach is to consider the space $\mC=[0,\infty)\times[0,\infty)$ of all possible parameters $(\alpha,\beta)$ and to analyse how the optimal paths change as $(\alpha,\beta)$ varies. A maximal subset of $\mC$ on which the set of optimal paths does not change is called an \emph{optimality region} of $\mC$. The shape of optimality regions in $\mC$ are semi-infinite cones bounded by the coordinate axes and by lines of the form $\beta = c+ \alpha(c + 1/2)$ for certain values of $c$. So it suffices to study the number of regions with one parameter fixed; we will set $\alpha=0$. 
	
	Denote the number of optimality regions in this model by $\alregmn$. Naturally the (expected) number of optimality regions attracted a lot of interest both theoretically \cite{Fer-Ven,Gus-Bala-Naor, Cynthia} and in applications \cite{MyersMiller88, Dew-Hug-Woo-Stu-Pac, Mal-Eri-Hug, How-Hei}. The current conjecture \cite{Fern-Sep-Slut, PachterSturmfels} is that $\E(\alreg nn) = O(\sqrt{n})$, but the complexity of the random variable does not allow for direct calculations. In this article we obtain an asymptotic lower bound for the optimal score when $a$ is fixed, as well as upper bounds for the number of optimality regions when the rectangle is of dimensions $m_n \times n$. With random words of this size the biological applications are unrealistic but the results offer some insight from a theoretical perspective. Moreover, we prove that $O(\sqrt{n})$ for the expectation is not the correct order in this case, at least for $a < 3/4$. 
	
Optimality regions can be studied in the independent model as well, and in fact we can obtain stronger results, again when the rectangle is of dimensions $m_n \times n$.

	\subsection{Outline} The paper is organised as follows: in Section \ref{sec:results} we state our main results. Section \ref{sec:ORmi} contains preliminary results that do not depend on the specific choice of environment and therefore hold for both the alignment and the independent model. The results concerning the independent model are proved in Section \ref{sec:ORBLIP} whereas in Section \ref{sec:ORalignment} we prove our results about the alignment model. 
		
	\subsection{Notation} We briefly collect the pieces of notation discussed so far and list the most common notation used in the paper. Letters $T$, $G$ and $L$ all denote last passage times: $T$ is for passage times under potential \eqref{eq:potcgm}, $G$ is the passage time for the independent model and $L$ its counterpart for the alignment model. 
The letter $R$ is reserved for the number of optimality regions, and we distinguish the regions in each of the two models by $\indreg{m}{n}$ the regions in the independent model, and by $\alreg{m}{n}$ the regions in the alignment model. We omit the superscripts when results hold for both models (see for example Section \ref{sec:ORmi}). 

Throughout, $p$ is a parameter in the interval $(0,1)$ and $q = 1-p$. $\mathcal A$ is the alphabet in the alignment model and $\abs\abc$ is its size.

\section{Results}		
\label{sec:results}
	In this section we have our main results, first for the independent model and then the softer ones for the alignment model. 
	
	\subsection{Independent model}
	See Section \ref{sec:ORBLIP} for a proof of Theorems \ref{thm:tightness},  \ref{thm:iid:regions:ub}, \ref{thm:TW} and Corollary \ref{cor:ORhalf}.
	
 We consider the last passage time $\bsz {m_n}n$ with $m_n=n/p-xn^a$ for suitably chosen $x$.  When the exponent $a$ is small we obtain tightness without rescaling, for any choice of $x$: 
	
	\begin{theorem}
		\label{thm:tightness}
		Let $x\in\R$ and $a\in \rb{0,\frac12}$. The sequence $\rb{n-\bsz{n/p-xn^a}{n}}_{n\in\N}$ is tight and
		\begin{align}
			\varlimsup_{n\to\infty} \P\set{n-\bsz{n/p-xn^a}{n}\geq k}\leq \begin{cases} 
												2^{-k}, &\quad a < 1/2\\
												\left(\Phi\left( xpq^{-\nicefrac12} \right)\right)^k, &\quad a = 1/2. 
											   \end{cases}
		\end{align}
		where $\Phi$ is the cumulative distribution function of the standard Gaussian distribution.
	\end{theorem}

	We will see in \eqref{eq:RnGbound} that $\indreg{m}{n}<n-\bsz mn$.  As a corollary we obtain an asymptotic bound on the expected number of optimality regions:
	
	\begin{corollary}
		\label{cor:ORhalf} Let $a\in\left(0,\nicefrac12\right]$. Then
		\begin{align}
			\varlimsup_{n\to\infty} \E \ab{\indreg{n/p-xn^a}n}\leq \begin{cases} 
												2, &\quad a < 1/2\\
												(1- \Phi\left( xpq^{-\nicefrac12} \right))^{-1}, &\quad a = 1/2. 
											   \end{cases}.
		\end{align}
	\end{corollary}

		\begin{figure} \label{fig:sim3}
			\centering
			    \includegraphics[scale=.25]{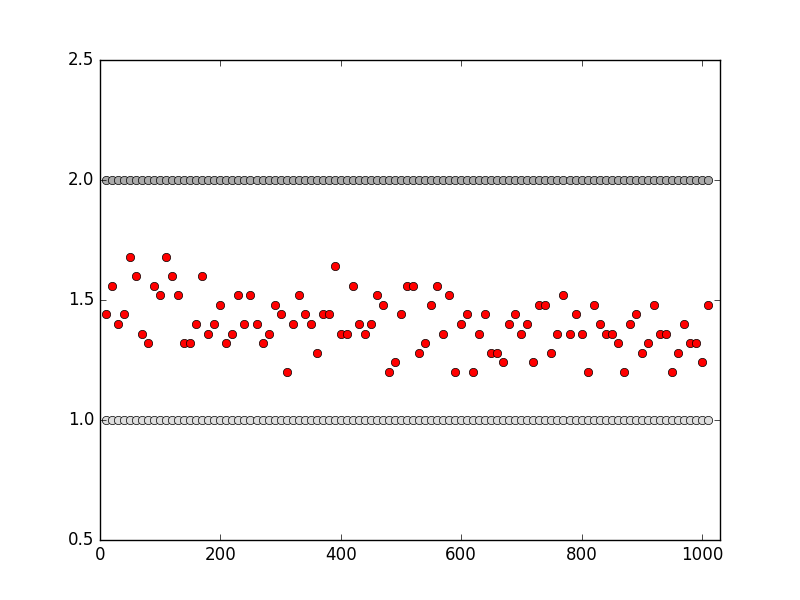} 
			\hspace{-0.5cm}
			    \includegraphics[scale=.25]{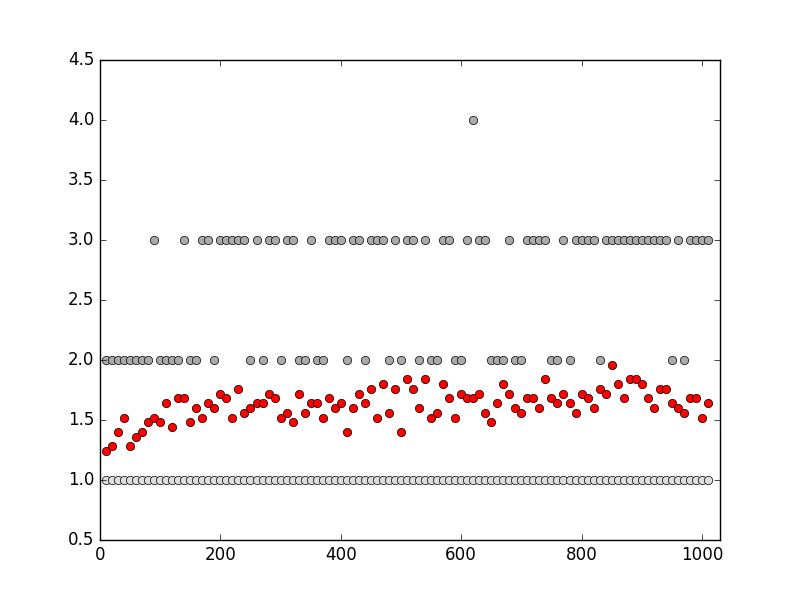}
			\hspace{-0.5cm}
		\includegraphics[scale=.25]{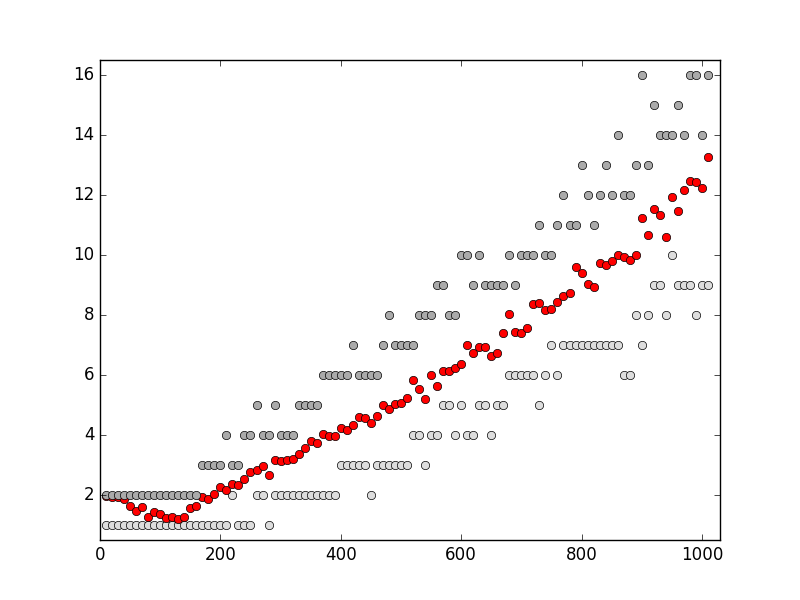}
		\caption{\footnotesize Monte Carlo simulations for the empirical maximum, minimum and expected number of regions for up to $n = 1000$ 
		in the independent model for $a = 0.8$ with varying $p = 0.05, 0.5, 0.8$ from left to right. 
		For each $n$, 25 independent environments were sampled. $n$ grows in increments of size $10$}
		\end{figure}

	For $a > 1/2$ we state a bound on the number $\indreg{m}{n}$ of optimality regions. 
	The optimal results and the relevant scaling of $m$ in terms of $n$ differ according to the value of $a$. 
	
	\begin{theorem}
	\label{thm:iid:regions:ub} 
	Let $a\in \rb{0,1}$.
	\begin{enumerate} 
	\item If $a \in (0, \nicefrac12]$,  
		\begin{align}
			\varlimsup_{n\to\infty} \P\set{\indreg{n/p - xn^a}{n}\geq k}\leq \begin{cases} 
												2^{-k}, &\quad a < 1/2\\
												\left(\Phi\left( xpq^{-\nicefrac12} \right)\right)^k, &\quad a = 1/2. 
											   \end{cases}
		\end{align}
	\item If $a\in(\nicefrac12,\nicefrac34]$ there exists a constant $C_1 = C_1(x, p)$ so that 
		 \be \label{eq:someregions}
			\lim_{n\to \infty} \P\Big\{ \indreg{n/p - xn^a}{n} > C_1n^{2a - 1}\Big\} = 0.
		\ee 
	\item If $a\in\rb{\nicefrac34,1}$ there exists a constant $C_2= C(x, p)$ so that, 
		 \be 
			\lim_{n\to \infty} \P\Big\{ \indreg{n/p - xn^a}n > C_2 n^{2a/3}\Big\} = 0.
		\ee 
	\end{enumerate}
	\end{theorem}
In the theorem above, equation \eqref{eq:someregions} holds also when $a > 3/4$, however the bound $n^{2a/3}$ is sharper. 

	Finally when $a\in( \nicefrac12,\nicefrac57]$ we obtain Tracy-Widom fluctuations. It is worth noting that we do not take the standard approach of scaling by the variance. Instead, we change the size of the rectangle, by subtracting a term of size $n^{\frac{2 - a}{3}}$ from the width.
	
	%
	
	\begin{theorem} 
	\label{thm:TW}
	For $s\in\R$ define $x = \frac{2}{\sqrt p}\left(\frac{q}{p}\right)^a$ and $y(s)=s\frac{\sqrt{p}}{q}\left(\frac{p}{q}\right)^{\frac{1+a}{3}}$. Then
	\begin{enumerate}
		\item For $1/2 < a < 2/3$,
	\be \label{eq:h1}
	\lim_{n \to \infty} \P \left\{ \bsz{\frac{n}{p} - x n^a - y(s) n^{\frac{2 - a}{3}}}{n} \le  n - \left(\frac{qn}{p}\right)^{2a-1}\right\} =F_{TW}(s).
	\ee
	\item For $2/3 \le a \le 5/7$,
	\be \label{eq:h2}
	\lim_{n \to \infty} \P \left\{ \bsz{\frac{n}{p} - x n^a - y(s) n^{\frac{2 - a}{3}}}{n} \le  n - \left(\frac{qn}{p}\right)^{2a-1} + ax \left(\frac{q}{p}\right)^{2a-2}n^{3a-2}\right\} =F_{TW}(s).
	\ee
	\end{enumerate}
	\end{theorem}

\begin{figure}[h]
			\begin{center}
				\begin{tikzpicture}[>=latex, yscale=0.65]
					\draw[<->] (0,4)--(0,0)--(12,0);
					\draw[fill, color=sussexg!60](6,0)rectangle(9,3.5);
					\draw[densely dotted, line width=1.6pt](6,0)rectangle(9,3.5);
					\draw(0,0)rectangle (11.5,3.5);
					\draw (-0.4, 3.5) node{$n$};
					\draw (11.5, -0.4) node{$\nicefrac np$};
					\draw (7.5, -0.4) node{$\nicefrac np-xn^a$};
					\draw[line width=2pt](7.5,0)--(7.5, 3.5); 
					\draw[|<-](6,4)--(6.5,4);
					\draw[->|](8.5,4)--(9,4);
					\draw(7.5, 4)node{$O(n^{\frac{2-a}{3}})$}; 
					\draw[color=nicosred,line width=2.6pt](6,3.5)--(9, 3.5);
				\end{tikzpicture}
			\end{center}
			\caption{\footnotesize Tracy-Widom fluctuations to the last passage time of the independent model 
			depend on position of the endpoint in the thickset red line. 
			When $a \in (\nicefrac12, \nicefrac 23)$ the Tracy-Widom reveals itself just by centering 
			according to the first and second order macroscopic approximation
			of the LLN for $G$. 
			However when $a \in (\nicefrac32, \nicefrac 57)$, a third order approximation to the law of large numbers, $c n^{3a-2}$, is necessary 
			for the Tracy-Widom fluctuations.}
		\end{figure}
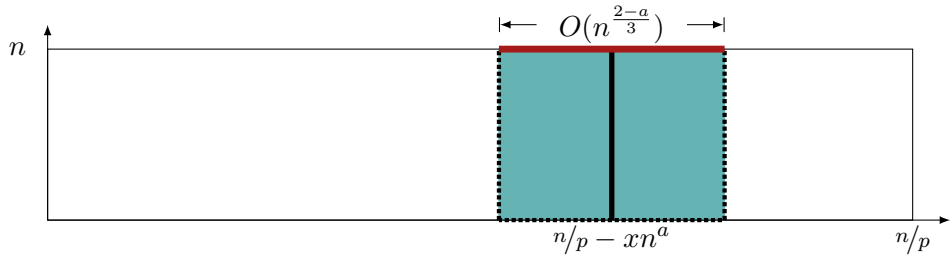

	\begin{remark}
		The case $a \ge \nicefrac57$ corresponds to an exponent $\gamma =2a-1 \ge \nicefrac37$ in equation \eqref{eq:B-M-TW} (see \cite{Bodineau-Martin})
		 and the result cannot be extended further with these techniques. In Section 3.1 of \cite{Bodineau-Martin} the authors explain why their result should 
		extend at least up to exponent $\gamma = \nicefrac34$. 
		The independent Bernoulli model here, while equivalent to the edge of the corner growth model  may be a bit more sensitive to these cut-offs and indeed $\gamma = \nicefrac37$ seems to be critical and manifests itself in the proof. 
		
		From the two cases of Theorem 
		\ref{thm:TW} we see that we need to amend the right-hand side of the event in  \eqref{eq:h1} by a term $O(n^{3a-2})$, in order to get the non-trivial result in \eqref{eq:h2}. This gives a new cut-off $a =\nicefrac23$ 
		(or $\gamma = \nicefrac13$). The term is there for case 2 as well, but when $a \le \nicefrac23$ the term is bounded and plays no role, while it must be dealt with, for higher $a$.
		
		Second, from the proof of Theorem \ref{thm:TW}, the exponent $a = \nicefrac57$ ($\gamma = \nicefrac37$) seems 
		to be critical, since it is necessary to have $2a-1 < \frac{2-a}{3}$ to balance the various orders of magnitude that appear.
		Assuming that the scaling in  $\eqref{eq:B-M-TW}$ remains the same for 
		$\gamma \in (\nicefrac37, \nicefrac34)$, this change implies a corresponding correction term of size $O(n^{\gamma})$ at the numerator of  \eqref{eq:B-M-TW}. \qed
		\end{remark}

	\subsection{Alignment model}
	Throughout we fix a finite alphabet $\abc$ with $\abs\abc\geq 2$, from which the letters of words $\wx$ and $\wy$ are chosen uniformly at random, independently of each other and let $a\in(0,1)$ and $\alpha,\beta\geq 0$. The proofs of Theorems \ref{thm:LPalignment}, \ref{thm:ORprobregions} and \ref{thm:ORexpregions} can be found in Section \ref{sec:ORalignment}. 
	
Define 
\[
g^{(a)}(n) = 
 \begin{cases}
 \sqrt{n \log n}, \quad a \le 1/2,\\
 n^a, \quad a > 1/2. 
 \end{cases}
\]
	\begin{theorem} \label{thm:LPalignment} Let $x > 0$ and $\beta \ge 0$.  For $\P$-a.e. $\omega$ we have the upper bound 
 \[
 \varlimsup_{n\to \infty} \frac{n ( 1 + \beta - \beta \abs\abc )-  \asz{\fl{n\abs\abc - xn^a}}n\beta}{g^{(a)}(n)} \le 
 \begin{cases} 
 \frac{\sqrt{2}}{\abs\abc -1} - \frac{1}{\abs\abc}, & a \le 1/2,\\
 \frac{1}{\abs\abc(\abs\abc-1)} - \beta x, & a > 1/2.
 \end{cases}
 \]
 and the lower bound
  \[
  \varliminf_{n\to \infty} \frac{n ( 1 + \beta - \beta \abs\abc )-  \asz{\fl{n\abs\abc - xn^a}}n\beta}{g^{(a)}(n)} \ge 
  \begin{cases}
  0, & a \le 1/2,\\
  - \beta x, & a > 1/2.
  \end{cases}
 \]
\end{theorem}

Finally, we turn to the number of optimality regions for the alignment model. The first result gives an upper bound on the asymptotic growth of the number regions:
	
	\begin{theorem} Let $x > 0$. There exists a constant $C_1(\abs\abc, x)$ that only depend on $x$ and $\abs\abc$ so that  
		\label{thm:ORprobregions}
		\be
\varlimsup_{n  \to \infty} \frac{\alreg{\fl{\abs\abc n - xn^a}}n}{(g^{(a)}(n))^{\nicefrac23}} \le C_1(x, \abs\abc), \quad \P-a.s.
\ee
The constant tends to $0$ as the alphabet size tends to $\infty$.
\end{theorem}
		
	We also have a bound of the same order for the expected number of optimality regions.
	\begin{theorem} 
		\label{thm:ORexpregions}
		There exists a constant $C_2(x, \abs\abc)$, 
		i
		so that
		\be\label{eq:exr}
		\varlimsup_{n\to \infty} \frac{ \E\big[\alreg{\abs\abc n-xn^a}n\big] }{(g^{(a)}(n))^{\nicefrac23} }\le C_2(x, \abs\abc).
		\ee
		The constant tends to $0$ as the alphabet size tends to $\infty$.
	\end{theorem}
	\begin{remark} These results are also valid for the independent model.
	Given the stronger bounds for the independent model,
	we do not expect \eqref{eq:exr} to be sharp, particularly for small values of the exponent $a$, and this is supported by Monte Carlo simulations. For example these suggest that for $a \le 1/2$ the number of 
	 expected regions is bounded (see Figure \ref{fig:sim1}). This is also the case for the independent model as we see in Theorem 
	 \ref{cor:ORhalf}.  For $a > 1/2$, the simulations in Figure \ref{fig:sim2} 
	 show that the expected number of regions is growing 
	 for small alphabet sizes, but again the exponent of growth is smaller than $2a/3$ and it seems to 
	 depend on the alphabet size.  
	\end{remark}
		\begin{figure}
		\begin{center}
		 \includegraphics[scale=.4]{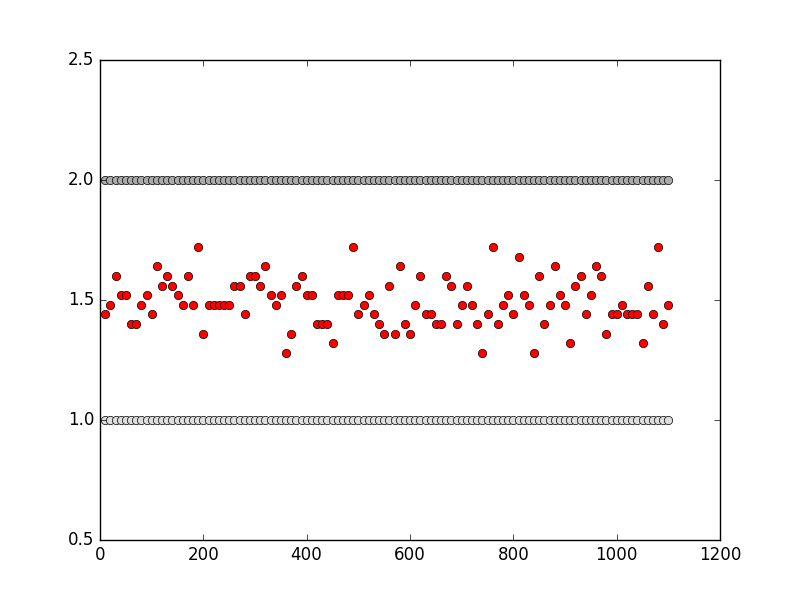}
		    \hspace{-0.5cm}
		        \includegraphics[scale=.4]{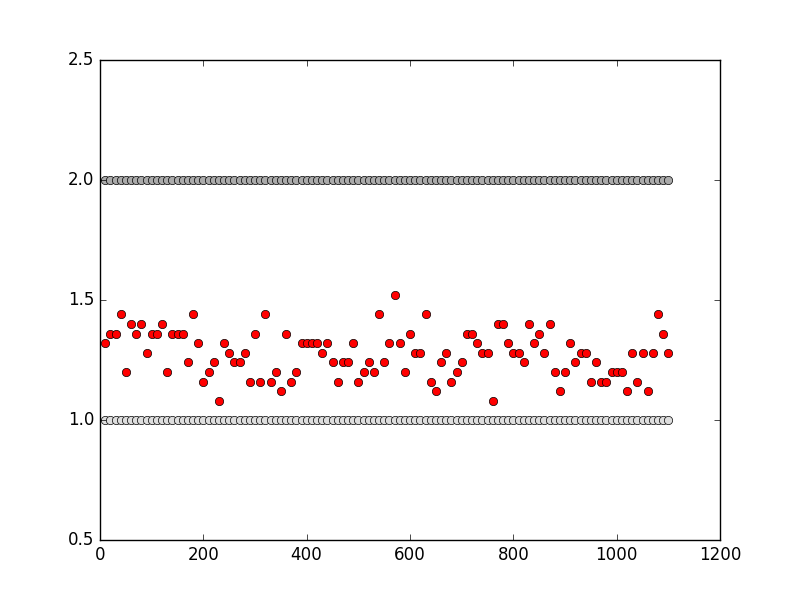}
		\end{center}
		\caption{ \footnotesize Monte Carlo simulations for the empirical maximum, minimum and expected number of regions for up to $n = 1000$ 
		in the alignment model for small values of $a$. For each $n$, 25 independent pairs of strings were uniformly chosen. $n$ grows in increments of size $10$ \newline
		(Left) $|\abc| = 20, a = 1/3, x=1$. (Right) $|\abc| = 2, a = 1/2, x=1$.  
		The simulations suggest the expected number of regions is bounded, and in agreement with the theoretical bound 
		obtained for the independent model.}
		  \label{fig:sim1}
	\end{figure}
	\medskip
	
	\begin{figure}
		\begin{center}	
		    \includegraphics[scale=.4]{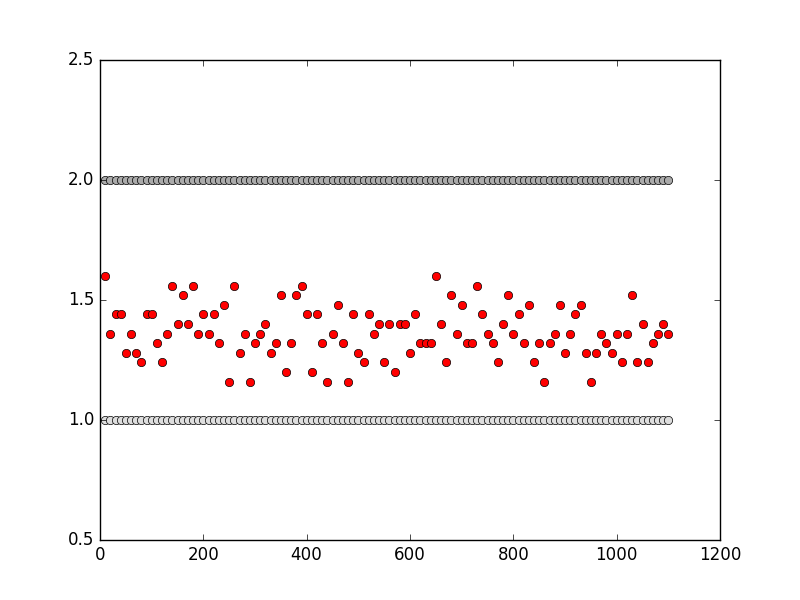}
		    \hspace{-0.5cm}
 		 \includegraphics[scale=.4]{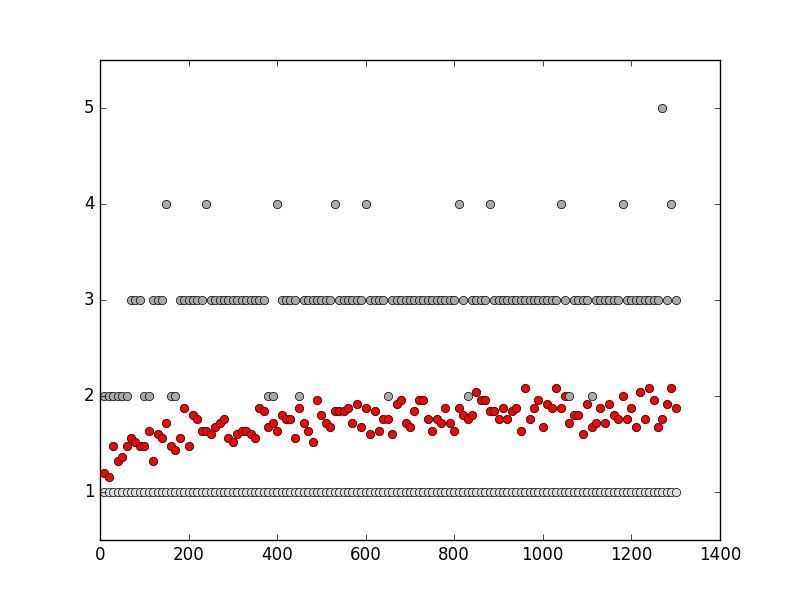}	
		\end{center}
		\caption{ \footnotesize Monte Carlo simulations for the empirical maximum, minimum and expected number of regions
		in the alignment model when $a$ is close to 1. For each $n$, 25 independent pairs of strings were uniformly chosen. $n$ grows in increments of size $10$.\newline
		(Left) $|\abc| = 20, a = 0.8, x=1$. (Right) $|\abc| = 2, a = 0.8, x=1$.  
		The simulations suggest that the expected number of regions is bounded for large alphabet sizes, but  for small size alphabets we see growth.}
		  \label{fig:sim2}
	\end{figure}

\section{ Model independent results for optimality regions and maximal paths}
	\label{sec:ORmi}
	In this section we present preliminary results about the two models that do not depend on the correlation structure of the weights. We therefore write $\regmn$ to mean either $\alregmn$ or $\indreg{m}{n}$. We also introduce the vocabulary usually used in the sequence alignment literature.
	
	Let $\pi=\set{u_0,\ldots,u_M}\in\Pi_{0, (m,n)}$ denote an admissible path and recall that the increments $z_k=u_k-u_{k-1}\in\mR=\set{e_1,e_2,e_1+e_2}$. Thus for each increment there are three possibilities:
				\begin{enumerate}
					\item $z_k = e_1+e_2$ with $\env_{u_k} = 0$, called a \emph{mismatch},
					\item $z_k \in\set{ e_1, e_2}$, called a \emph{gap},
					\item $z_k = e_1+e_2$ with $\env_{u_k} = 1$, called a  \emph{match}.
				\end{enumerate}
			Let $x=x(\pi)$ be the number of mismatches, $y=y(\pi)$ the number of gaps and $z=z(\pi)$ 
			the number of matches of $\pi$. 
			We also denote this triplet by
			$\mathbf s(\pi) = (x(\pi), y(\pi), z(\pi)).$
				
			Fix parameters $\alpha,\beta\geq 0$. Under potential $V_{\alpha, \beta}$ the score of the path $\pi$ is then given by 
				\be \label{eq:path:score}
						w_{\alpha, \beta}(\pi)  = z -\alpha x - \beta y.
				\ee
			Since any diagonal step is equivalent to an $e_1$ step followed by a $e_2$ step or vice versa, we have
			\be\label{eq:path:steps}
				m+n = 2x(\pi) +2z(\pi) + y(\pi) \quad\quad \text{ for all } \pi\in\Pi_{0, (m,n)}.
			\ee
			The last passage time $\is mn\alpha\beta$ (or $\as mn\alpha\beta$, depending on the environment) under potential defined in \eqref{eq:vab} can now be rewritten as
				\be \notag
				\is mn\alpha\beta = \max_{ \pi \in \Pi_{0, (m,n)}} \{ w_{\alpha, \beta} (\pi) \}.
				\ee	
				Our focus will be on the \emph{minimal-gap maximisers (MGM)}: paths whose score attains the last passage time with the smallest possible number of gaps. Since any two MGM paths have the same number of gaps and the same score it follows from \eqref{eq:path:steps} that
			\begin{lemma} \label{lm:3.2}
						All MGM paths have the same number of gaps, matches and mismatches.
			\end{lemma} 

			We denote the set of MGM paths by  $\pimax$. When $\alpha=0$
			 we write $\pimo\beta$.

			\begin{definition}
				Two points $(\alpha_1, \beta_1)$ and $(\alpha_2, \beta_2)$ belong in different
				\emph{optimality regions of the parameter space} for a fixed terminal point $(m,n)$ if and only if
				$\pim{\alpha_1}{\beta_1} \cap \pim{\alpha_2}{\beta_2}= \varnothing$. 
			\end{definition}
		
			
			For future reference we record the following observations:

			\begin{enumerate}
				\item For fixed $\alpha\geq 0$ and any $\beta_1 \le \beta_2$ we have 
				\begin{align}
					\label{eq:monotonicity}
						 w_{\alpha, \beta_1} (\pi) &\ge  w_{\alpha, \beta_2} (\pi)
				\end{align}
				and therefore this inequality also holds for the passage times:
				\begin{align}		 
						 \is mn\alpha{\beta_1} \ge \is mn\alpha{\beta_2}\quad\quad\text{and}\quad\quad \as mn\alpha{\beta_1} \ge \as mn\alpha{\beta_2}
				\end{align}
					
				\item For $\alpha = -1$ and $\beta = -1/2$, the weight of any path $\pi\in\Pi_{0,(m,n)}$ is given by
						\be
							w_{-1, -1/2}(\pi) = \frac{m+n}{2}
						\ee
			\end{enumerate}

			\begin{lemma}
				\label{lem:cones}
					All optimality regions in the $(\alpha, \beta)$-positive quadrant are semi-infinite 
					cones bounded by the coordinate axes and lines of the form
					$\beta = c+ \alpha(c + \nicefrac12)$.
			\end{lemma}
			
			This result was first proved in \cite{Gus-Bala-Naor}; we give a simplified proof here:
		
			\begin{proof}
		 Pick any $(\alpha, \beta) \in \R^2_+$ and let $(0, \beta')$ be the point
		 of intersection  of the linear segment connecting $(\alpha, \beta)$ and $(-1, -1/2)$ with the $y$-axis, i.e.
		 \be \label{eq:beta:lin}
		 	\beta = (\alpha+1)\beta'+ \frac{\alpha}{2}.
		 \ee
		 We will show that the optimal paths associated with $(0, \beta')$ are the same as those associated to $(\alpha, \beta)$. 
		 Consider any $\pi\in\Pi_{0,(m,n)}$ with  $\mathbf{ s }(\pi)=(x,y,z)$. Then
		 \begin{align}
		 w_{\alpha, \beta}(\pi) \notag
		 	&= z - x \alpha - y \beta = z - x\alpha - y\beta + y\beta' - y\beta' \notag 
		 	= w_{0,\beta'}(\pi) - x \alpha - (\beta - \beta')y \notag \\
			&=  w_{0,\beta'}(\pi) - x \alpha - \Big( (\alpha+1)\beta'+ \frac{\alpha}{2} - \beta'\Big)y, \quad \text{ by \eqref{eq:beta:lin},} \notag \\
			&=  w_{0,\beta'}(\pi) - \alpha\rb{ \frac{m+n}2 - w_{0, \beta'}(\pi)}, \quad \text{ by \eqref{eq:path:steps},} \notag \\
			& = (1+\alpha)  w_{0,\beta'}(\pi)  - \alpha \frac{m+n}2. \label{eq:shift}
		 \end{align}
		 So the weight of any path with  parameters $(\alpha, \beta)$ is an affine function
		  of the weight with parameters $(0,\beta')$ and
		 the two parameters
		 must belong to the same optimality region.
	\end{proof}

	\par\noindent Under a fixed environment $\env$, we define the \emph{critical penalties}
		\be\label{eq:cb}
			0 < \beta_1 < \cdots < \beta_{\regmn} < \infty 
		\ee
	to be the the gap penalties for $\alpha =0$ at which the optimality region changes. 
	We will also write $\beta_\infty$ for the last threshold $\beta_{\regmn}$.

	\begin{lemma}[Critical penalties]
	\label{lem:critical}
	For each $k \le \regmn$ let $\pi^{(\beta_k)} \in \pimo{\beta_k}$, with 
	$\mathbf{s}(\pi^{(\beta_k)}) = (x_{\beta_k}, y_{\beta_k}, z_{\beta_k})$. Then  
	\begin{align}
	\label{eq:betaK}
	\beta_{k+1} = \frac{z_{\beta_k} - z_{\beta_{k+1}}}{y_{\beta_k} - y_{\beta_{k+1}}}.
	\end{align}
	\end{lemma}

	\begin{proof}

		Continuity of the optimal score in the parameter $\beta$ implies that at $\beta_{k+1}$ 
		the weights will be the same whether $\beta_{k+1}$ is approached by above (considering 
		scores of paths in $\pimo{\beta_{k+1}}$) or from below (scores of paths in $\pimo{\beta_{k}}$). 
		Therefore
		\[
			z_{\beta_{k}} - \beta_{k+1} y_{\beta_k} = z_{\beta_{k+1}} - \beta_{k+1} y_{\beta_{k+1}} 
		\] 
		which yields the conclusion.
		\end{proof}
	
	\par\noindent Upper bounds for  the maximal value of $\regmn$ can be found in \cite{Fern-Sep-Slut}. For the LCS model these are sharp when the alphabet size grows to infinity. 
	The results and arguments in \cite{Fern-Sep-Slut} can be extended
	to give the upper bound 	
	\be
	 	\reg{\fl{ns} + o(n)}{\fl{nt} +o(n)} \le Cn^{2/3},
	\ee
	that holds in any fixed realization of the environment, any $(s,t)\in \R^2_+$ and $n$ 
	large enough. They also proved that environments that actually generate so many regions 
	exist, at least when the alphabet size was infinite. 
	This was later verified also for finite alphabets in \cite{Cynthia}.
	 
	\begin{lemma}
		\label{lem:naivebound}
		For $\beta_0 =0$ and each critical $\beta_k$ in \eqref{eq:cb}, choose an MGM path $\pi_k \in \pimo{\beta_{k}}$ with 
	$\mathbf{s}\rb{\pi_k} = (x_k, y_k, z_k)$ for $0 \le k \le \regmn$. Then
		\be
		\label{eq:naivebound}
		\regmn \le \min\Big\{ z_0 - z_{\regmn}, \frac{x_{\regmn}-x_0}{2}, \frac{y_0 - y_{\regmn}}{2}, n\wedge m - z_0 \Big\}.
		\ee
	\end{lemma}

	\begin{proof} 
	Distinct paths $\pi_i$ differ in the number of diagonal steps and the number of gaps. 
	Since a diagonal step is equivalent to two gaps,
	we have 
	$ 
	y_{i} - y_{i+1} \ge 2.
	$
	Furthermore it must be the case that 
	$ 
	z_i -  z_{i+1} \ge1
	;
	$
	otherwise $\pi_i$ would violate the MGM condition. Equation \eqref{eq:path:steps} and the last two inequalities give
	$ 
	x_{i+1} - x_i \ge 2.
	$
	Adding each inequality over $i$ gives the first three terms in the minimum of \eqref{eq:naivebound}.
	For the last term note that  $y_{\regmn} = n\vee m-m\wedge n$. Since $x_0 \ge 0$ \eqref{eq:path:steps}  yields 
		$2(n\wedge m - z_0 ) \ge y_0 - y_{\regmn} $. 
	\end{proof}

\begin{remark}
	\label{rmk:RnGbound}
	Notice that the last bound in \eqref{eq:naivebound} can be written as 
	\begin{align}
		\label{eq:RnGbound}
		\alregmn \leq n-\asz mn0 \quad\quad\text{and}\quad\quad\indreg{m}{n}  \leq n - \bsz mn. 
	\end{align}\qed
\end{remark}
	
Finally, we present a lemma that gives a useful bound on the number of regions if a bit more information is available. 

\begin{lemma} \label{lem:future}
Let $m = m(n)$ so that $m(n) \to \infty$ as $n \to \infty$. Let $g(n)$ be a deterministic function so that $\lim_{n\to \infty} g(n) = \infty$. Then, there exists an $N >0$ and a non-random constant $C_0$ so that  for all $n > N$ we have the inclusion of events 
\be \label{eq:epiphany}
A_n = \{ z_0 - z_R + y_0 - y_R \le g(n)\} \subseteq \{ R_{m,n} \le C_0 (g(n))^{\nicefrac23}\}.
\ee

In particular, 
\begin{enumerate}
\item If $\P\{ A_n ^c \text{  i.o  } \} = 0$, then
the number of optimality regions $R_{m,n}$ satisfies 
\be \label{eq:q1}
\varlimsup_{n  \to \infty} \frac{R_{m,n}}{g(n)^{\nicefrac23}} \le C_0, \quad \P-a.s.
\ee

\item  If $\P\{ A_n \} \to 1$, then the number of optimality regions $R_{m,n}$ satisfies 
\be \label{eq:q2}
\lim_{n  \to \infty} \P\Big\{ \frac{R_{m,n}}{g(n)^{\nicefrac23}} \le C_0 \Big\}=1.
\ee
\end{enumerate}
\end{lemma}	
	
\begin{proof}
Statements \eqref{eq:q1}, \eqref{eq:q2} are immediate corollaries of \eqref{eq:epiphany} which we now show. Fix an environment $\omega \in A_n$. Then we have that 
\[
z_0 - z_R + y_0 - y_R = \sum_{i=0}^{R_{m,n} - 1} \{(z_{\beta_{i+1}}-z_{\beta_i}) + (y_{\beta_{i+1}} - y_{\beta_{i}}) \}\le g(n).  
\]
	The sum above has as terms the numerators and denominators of the critical penalties (see Lemma \ref{lem:critical}). 
	
	Each critical penalty is a distinct rational number and it corresponds to a change of optimality region. The bound $g(n)$ is independent of the environment,  so we can obtain an upper bound on the number of regions that is independent of the environment, if we maximize the number of terms that appear in the sum. 
	
	Since the terms in the sum are integers, the maximal number of terms is the maximal number of integers $k$ that can be added so that the bound $g(n)$ is not exceeded. Those integers $k$ need not be distinct but they need to able to be written as a sum of integers $a, b$, $k = a+b$ so that $a/b$ are different. This is because the ratio $a/b$ corresponds to critical penalties and those are distinct. Take each successive integer $k$ and compute the number of irreducible fractions $a/b$ so that $a+b=k$.
	
	 The number of irreducible fractions satisfying this is $\varphi(k)$, where $\varphi$ is Euler's totient function \cite{Apostol}. 
	 The number of distinct values $k$ that can be used is $M_{\max}$, which must satisfy 
	  \[
	 \sum_{k= 1}^{M_{\max}} k \varphi(k) \le g(n) <  \sum_{k = 1}^{M_{\max}+1} k \varphi(k).
	 \]
	  These inequalities imply that 
	  $M_{\max}$ will be bounded above, up to a lower order term, by $cg(n)^{1/3}$. 
	 This follows by the asymptotics of $\varphi$ for large arguments, 
	 and we direct the reader to the proof of Theorem 5 in \cite{Fern-Sep-Slut} for the details. 
	 The bound on $M_{\max}$ is true for all $n > N_1$ large enough.  
	 Then an upper bound for the number of admissible pairs $(a, b)$ 
	 (and therefore for the maximal number of regions) is 
	 \[\sum_{k=1}^{M_{\max}} \phi(k) \le  c_1 M_{\max}^2 \le C g(n)^{2/3}.\]
	 This last estimate is again the result of an analytic number theory formula 
	 (see \cite{Apostol}) which also works for $n > N_2$ large enough. 
	 So both deterministic bounds hold for all 
	 $n > N = N_1 \vee N_2$.  
\end{proof}	
	
	The difficulty with the alignment model is the correlated environment. Therefore, 
	the soft techniques below try to avoid precisely this issue. The same techniques work for 
	the BLIP model and give identical bounds, but the exact solvability of that model often allows 
	sharper results.
	
	Our strategy is to construct a path with a score that is near-optimal under any penalty $\beta$ 
	and which attempts to minimize as much as possible the number of vertical steps. 
	This will be important for the lower bound for the passage time under penalty $\beta_R$,
	 where we know that the optimal path takes no vertical steps. We present the construction 
	 and results for alignment model, but re-emphasize that they hold for both.
	
	\subsection{Construction of the path} \label{ssec:3.1}
	Fix an environment $\omega$ on $\N^2$, defined by two infinite words $\wx$, $\wy$,
	where each letter is chosen uniformly at random. $\omega_{i,j}$ is defined according to 
	\eqref{eq:alenv}.
	
	 Consider the following strategy $(S)$ to create a path $\pi_S$: 
	 \begin{enumerate} 
		 \item For some appropriate constants $c_1$ and $c_2$ (to be determined later), move with $e_1+e_2$ steps from $0$ up to a fixed point 
		 \begin{align}
			 \label{eq:DefUn}
			 u_n(a) = \begin{cases}
			 \vspace{0.1in}
			 \left(\fl{\sqrt{c_1 n \log n}}, \fl{\sqrt{c_1 n \log n}}\right), \quad & \text{if } a \le \nicefrac12,\\
			 \left(\fl{\frac{1}{\abs\abc -1} xn^a} + \fl{ \sqrt{c_2 n\log n}}, \fl{\frac{1}{\abs\abc -1} xn^a } + \fl{\sqrt{c_2 n\log n}}\right), \quad & \text{if } a > \nicefrac12.
			 \end{cases}
		\end{align}
		\item Now, from $u_n(a)$ construct the path as follows
	 \begin{enumerate}
		\item If the path is on site $(i,j)$ with $j < n$ and $\om_{i+1,j+1} = 1$ 
			then move diagonally with an $e_1+e_2$ step, and now the path is on site $(i+1, j+1)$.
		\item If the path is on site $(i,j)$ with $i < \fl{\abs{\abc}n - xn^a}$ and $\om_{i+1,j+1} = 0$ 
			then move horizontally with an $e_1$ step, and now the path is on site $(i+1, j)$.
		\item If $j = n$ or $i = \fl{\abs{\abc}n - xn^a}$, move to 
		$(\fl{\abs{\abc}n - xn^a}, n)$.
	\end{enumerate}
\end{enumerate}

From this description it is not clear whether we can enforce the condition that no vertical steps will be taken by $\pi_S$. However, this will happen for eventually all $n$, by choosing constants $c_1, c_2$ appropriately. 
Consider an infinite path $\bar \pi_S$ that moves according to strategy $(S)$ but without the restrictions 
$i < \fl{\abs{\abc}n - xn^a}$ for (3)-(b) and without step (3)-(c). 

Let $Y_j$ be the random variables that give the amount of horizontal steps path $\bar \pi_S$ takes at level $y = j + u_n(a)\cdot e_2$,
\be \label{eq:Yincrements0}
Y_j = \abs{\{ i \in \N: (i,j + u_n(a)\cdot e_2) \in \bar \pi_S \}}.
\ee
Because $\bar \pi_S$ does not have a target endpoint, we have
\be \label{eq:Yincrements}
Y_j  \sim \text{Geom}\big( \nicefrac1{\abs{\abc}}\big), \quad \P\{ Y_j = \ell\} =\nicefrac{1}{\abs\abc}(1 -  \nicefrac{1}{\abs\abc})^{\ell -1}.
\ee
By construction, the $Y_j$ are i.i.d.~ with mean $\abs\abc$.

Path $\bar \pi_S$ coincides with $\pi_S$ up until the point that $\bar \pi_S$ hits either the north or east boundary of the rectangle $[0, \fl{\abs{\abc}n - xn^a}] \times [0, n]$. When $\bar \pi_S$ touches the north boundary first, we can conclude that $\pi_S$ has no vertical steps up to that point.
We will estimate precisely this probability, using the following moderate deviations lemma \cite{Cramer38}.
\begin{lemma}
	\label{lem:MD}
	Let $\rb{X_N}_{N\in\N}$ an i.i.d. sequence of random variables with exponential moments. If  $N\lambda_N^2 \to \infty$ and $ N\lambda_N^3 \to 0$ then
	\be \label{eq:mdas}
	\P\Big\{ \Big|\frac1N\, \sum_{i=1}^N X_i - \E(X_1)\Big| > \lambda_N \Big\} \sim \frac{2}{\sqrt{2\pi N \lambda_N^2}}e^{-N\lambda^2_N/2}.
	\ee
\end{lemma}

From the equality of events 
\be \label{eq:eventeq}
\{ \bar \pi_S \text{ exits from the north boundary} \} = \Big\{ \sum_{j=1}^{n - u_n(a)\cdot e_2 }Y_j  \le \fl{\abs{\abc} n- xn^a} - u_n(a)\cdot e_1\Big\},
\ee
we estimate for $a \le 1/2$ and for $n$ sufficiently large for the asymptotics in \eqref{eq:mdas} to be accurate,
\begin{align}
			 \P\Big\{ \sum_{j=1}^{n - u_n(a)\cdot e_2 }Y_j &\le \fl{\abs{\abc}n - xn^a} - \fl{\sqrt{c_1 n \log n}}\Big\} \label{eq:grr}\\
				&\ge \P\bigg\{ \sum_{j=1}^{n - \fl{\sqrt{c_1 n \log n}} }(Y_j - \abs{\abc}) \le  - xn^a + (\abs{\abc} - 1)\sqrt{c_1n \log n} -3 \bigg\}\notag \\
				&\ge 1 - c_0  \frac{1}{\sqrt{\log n} }n^{-c_1 (\abs{\abc}-1)^2  /4}. \notag
		\end{align}
For the last inequality, we used Lemma \ref{lem:MD} for 
\[ N = n - \fl{\sqrt{c_1 n \log n}} \text{ and }  \lambda_N = (\abs\abc -1)\sqrt{c_1} \sqrt{\frac{\log n}{n}} + O(n^{\alpha - 1}).\]
The constant $c_0$ only depends on $\abs\abc$ which is assumed to be strictly larger than 1. Choose $c_1 > \frac{2}{(\abs\abc -1)^2}$ so that the probabilities of the event $\{ \bar \pi_S \text{ exits from the east boundary} \}$ are summable in $n$. Then by the Borel-Cantelli lemma, we can find an $M = M(\omega)$ so that for all $n> M$ path $\bar\pi_S$ hits the north boundary first. 

The situation for $a > 1/2$ is similar. Starting from \eqref{eq:grr}, we have 
\begin{align}
			 \P\Big\{ \sum_{j=1}^{n - u_n(a)\cdot e_2 }Y_j &\le \fl{\abs{\abc}n - xn^a} - \fl{\frac{1}{\abs\abc -1}xn^a}- \fl{\sqrt{c_2 n \log n}}\Big\} \label{eq:grr2}\\
				&\ge \P\bigg\{ \sum_{j=1}^{n - \fl{\frac{1}{\abs\abc -1}  xn^a} -\fl{\sqrt{c_2 n \log n}} }(Y_j - \abs{\abc}) \le  (\abs{\abc} - 1)\sqrt{c_2n \log n} -3 \bigg\}.\notag 
		\end{align}
 Then the proof goes as for the previous case, and again it suffices that  $c_2 > \frac{2}{(\abs\abc -1)^2}$. 
 
 From the definition of $\pi_S$ and the above discussion, we have shown the following: 

 \begin{lemma} \label{lem:pis}For $\P$- a.e.~$\om$ there exists $M=M(\omega)$ so that for all $n > M(\omega)$, path $\pi_S$ exits from the north boundary of the rectangle $[0, \abs\abc n - xn^a] \times [0,n]$. In that case, 
 \begin{enumerate}
 \item it has no vertical gaps until the point of exit, 
 \item the number of horizontal gaps it has is  $(\abs\abc -1)n - xn^a$ (the minimal possible), and 
 \item it collects $n - u_n(a) \cdot e_2 + \sum_{k=1}^{u_n(a)\cdot e_2} \om_{k,k}$ positive weight. 
 \end{enumerate}
 Since $\pi_S$ has the smallest number of gaps possible, it can be optimal under any penalty $\beta$.
 \end{lemma}

	\section{The independent model}
	\label{sec:ORBLIP}
	In this section we prove results about the independent model. We begin with a coupling between the longest common subsequence in the independent model, with 
	the corner growth model in an i.i.d.~$\text{Geom}(1-p)$ environment. This is achieved via the following identity.
	 Recall that
	$\cgs mn$ denotes the last passage time 
	in an $m\times n$ rectangle, with admissible $e_1$ or $e_2$ steps only, under potential \eqref{eq:potcgm}.	 
	\begin{align}
	\label{eq:disteq}
		\P\big\{ G^{(0)}_{m,n} \le m - N \big\} = \P\set{ \cgs{n - m + N}N \le n + N -1}.		
	\end{align}
	The result follows from the arguments in \cite{Geo-2010}, and we briefly present the main idea.
	
	The \emph{discrete totally asymmetric simple exclusion process} (DTASEP) with backward updating is an interacting particle system of left-finite particle configuration on the integer lattice, i.e. such that sites to the left of some threshold are empty  (see  Figure \ref{fig:pic2}). Label the particles from left to right and denote the position of the $j^\text{th}$ particle at time $\ell\in\N$ by $\eta_j(\ell)$. At every discrete time step $\ell \in \N$ each particle independently attempts to jump one step to the left with probability $q=1-p$. Particle $i$ performs the jump if either 
\begin{enumerate} 
\item the target site was unoccupied by particle $i-1$ at time $\ell-1$ 
or,
\item the target site was occupied by particle $i-1$, but it also performs a jump at time $\ell$. 
\end{enumerate} 
In words, particles are forbidden to jump to occupied sites and we update  from left to right. Start DTASEP with the \emph{step initial condition} $\eta_i(0) = i$ so that initially the $i$-th particle is at position $i$. Let $\tau_{i,j}$ be the time it takes particle $j$ to jump $i$ times:
\[
	\tau_{i,j} =  \inf\{ \ell \ge 0 : \eta_j(\ell)\le j-i \} .
\]
Then the following recursive equation holds 
\[
\tau_{i,j} =\tau_{i, j-1}\vee (\tau_{i-1,j} +1) + \tilde\zeta_{i,j}.
\]
where the $\tilde\zeta_{i,j}$ are independent Geometric variables with parameter $q=1-p$, supported on $\N_0$.

\begin{figure}
\begin{center}
\begin{tikzpicture}[>=latex, scale=0.7]
%

\draw[color=blue, line width=2pt](6.5-7, 7)--(5.5-6, 6)--(5.5-4,4)--(4.5-3, 3)--(4.5-1,1)--(3.5, 0);


\foreach  \x in {1,..., 7}
{
\draw(\x, 0)--(\x-7.5, 7.5);
\draw(\x-0.5, -0.5)node{$\x$};
\draw(-\x, \x)--(7.5-\x,\x);
\draw(-\x-0.8, \x)node{$t = \x$};
}

\draw[<->, gray](-8,8)--(0,0)--(8,0);


\draw[color=nicosred](3.5-0.5, 0.5)node{$\otimes$};
\draw[color=nicosred](5.5-0.5, 0.5)node{$\otimes$};

\draw[color=nicosred](2.5-1.5, 1.5)node{$\otimes$};
\draw[color=nicosred](3.5-1.5, 1.5)node{$\otimes$};
\draw[color=nicosred](6.5-1.5, 1.5)node{$\otimes$};

\draw[color=nicosred](1.5-2.5, 2.5)node{$\otimes$};
\draw[color=nicosred](2.5-2.5, 2.5)node{$\otimes$};

\draw[color=nicosred](3.5-3.5, 3.5)node{$\otimes$};

\draw[color=nicosred](0.5-4.5, 4.5)node{$\otimes$};
\draw[color=nicosred](3.5-4.5, 4.5)node{$\otimes$};
\draw[color=nicosred](6.5-4.5, 4.5)node{$\otimes$};

\draw[color=nicosred](2.5-5.5, 5.5)node{$\otimes$};
\draw[color=nicosred](6.5-5.5, 5.5)node{$\otimes$};

\draw[color=nicosred](-6, 6.5)node{$\otimes$};
\draw[color=nicosred](1.5-6.5, 6.5)node{$\otimes$};
\draw[color=nicosred](4.5-6.5, 6.5)node{$\otimes$};


\foreach \x in {0,...,6}{
\shade[ball color=red](\x+0.5, 0)circle(1.3mm);
}

\foreach \x in {0,...,2, 4,5, 6}{
\shade[ball color=red](\x+0.5-1, 1)circle(1.3mm);
}

\foreach \x in {0, 1, 3, 4,5}{
\shade[ball color=red](\x+0.5-2, 2)circle(1.3mm);
}

\foreach \x in {0, 2, 3, 4,5}{
\shade[ball color=red](\x+0.5-3, 3)circle(1.3mm);
}

\foreach \x in {0, 2, 4,5, 6}{
\shade[ball color=red](\x+0.5-4, 4)circle(1.3mm);
}

\foreach \x in {1, 2, 4,5}{
\shade[ball color=red](\x+0.5-5, 5)circle(1.3mm);
}

\foreach \x in {1, 3, 4,5}{
\shade[ball color=red](\x+0.5-6, 6)circle(1.3mm);
}

\foreach \x in {2, 3, 5, 6}{
\shade[ball color=red](\x+0.5-7, 7)circle(1.3mm);
}

\end{tikzpicture}
\caption{Space-time realisation of DTASEP (Graphical construction). Particles move to the left, according to exclusion rules (1) and (2). Symbols $\otimes$ denote Bernoulli($p$) weights 1, and particle underneath an $\otimes$ symbol cannot jump during that time, i.e. particles jump with probability $1-p=q$ as long as the exclusion rule is not violated. The trajectory of particle 4 is highlighted for reference.}
\label{fig:pic2}
\end{center}
\end{figure}
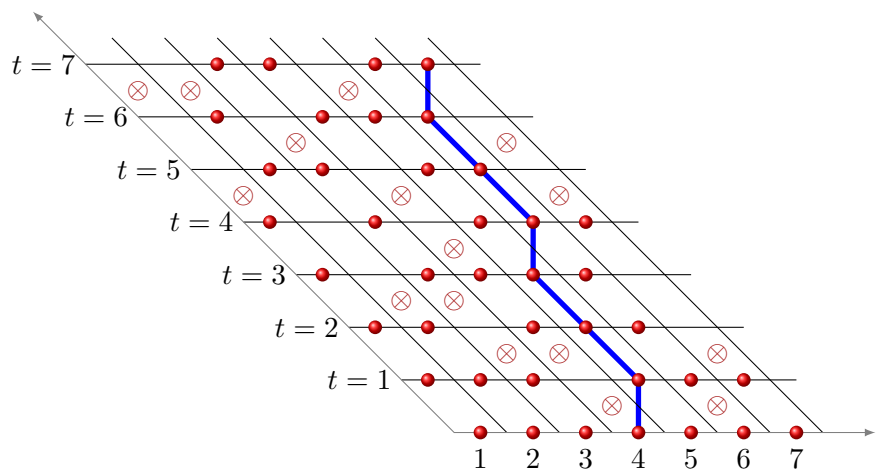
%
By setting $\zeta_{i,j} = \tilde \zeta_{i,j} +1 \sim \text{Geom}(1-p) \in \{ 1, 2, \ldots \}$,  the $\tau_{i,j}$ 
can be coupled with the last passage time in the corner growth model (cf. \cite{Geo-2010}, Lemma 5.1), 
giving the equality in distribution
\be \label{eq:bliptoG}
\tau_{i,j} \stackrel{(\mathrm d)}{=} \cgs ij - j +1.
\ee

We embed DTASEP in the two-dimensional lattice $\Z \times \N_+$, 
using its graphical construction as follows: Let $\set{b_{k,\ell}\colon (k,\ell)\in \Z \times \N_+}$ be a field of i.i.d. Bernoulli$(q)$ random variables and assign to each site $(k,\ell)$ the random weight $b_{k,\ell}$.
Particles are placed initially on $\N_+ \times\{ 0\}$, with particle $i$ at coordinate $(\eta_i(0), 0)$. 
The Bernoulli marked sites signify which particles will attempt to jump in the DTASEP process. 

\begin{figure}[h]
	\begin{center}
	\begin{tikzpicture}[>=latex, scale=0.8]

\draw[color=blue, line width=2pt](6.5, 7)--(5.5, 6)--(5.5,4)--(4.5, 3)--(4.5,1)--(3.5, 0);


\foreach  \x in {1,..., 7}
{
\draw(\x, 0)--(\x, 7.5);
\draw(\x-0.5, -0.5)node{$\x$};
\draw(0, \x)--(7.5,\x);
\draw(-0.8, \x)node{$t = \x$};
}

\draw[<->](0,8)--(0,0)--(8,0);


\draw[color=nicosred](3.5, 0.5)node{$\otimes$};
\draw[color=nicosred](5.5, 0.5)node{$\otimes$};

\draw[color=nicosred](2.5, 1.5)node{$\otimes$};
\draw[color=nicosred](3.5, 1.5)node{$\otimes$};
\draw[color=nicosred](6.5, 1.5)node{$\otimes$};

\draw[color=nicosred](1.5, 2.5)node{$\otimes$};
\draw[color=nicosred](2.5, 2.5)node{$\otimes$};

\draw[color=nicosred](3.5, 3.5)node{$\otimes$};

\draw[color=nicosred](0.5, 4.5)node{$\otimes$};
\draw[color=nicosred](3.5, 4.5)node{$\otimes$};
\draw[color=nicosred](6.5, 4.5)node{$\otimes$};

\draw[color=nicosred](2.5, 5.5)node{$\otimes$};
\draw[color=nicosred](6.5, 5.5)node{$\otimes$};

\draw[color=nicosred](0.5, 6.5)node{$\otimes$};
\draw[color=nicosred](1.5, 6.5)node{$\otimes$};
\draw[color=nicosred](4.5, 6.5)node{$\otimes$};


\foreach \x in {0,...,6}{
\shade[ball color=red](\x+0.5, 0)circle(1.3mm);
}

\foreach \x in {0,...,2, 4,5, 6}{
\shade[ball color=red](\x+0.5, 1)circle(1.3mm);
}

\foreach \x in {0, 1, 3, 4,5}{
\shade[ball color=red](\x+0.5, 2)circle(1.3mm);
}

\foreach \x in {0, 2, 3, 4,5}{
\shade[ball color=red](\x+0.5, 3)circle(1.3mm);
}

\foreach \x in {0, 2, 4,5, 6}{
\shade[ball color=red](\x+0.5, 4)circle(1.3mm);
}

\foreach \x in {1, 2, 4,5}{
\shade[ball color=red](\x+0.5, 5)circle(1.3mm);
}

\foreach \x in {1, 3, 4,5}{
\shade[ball color=red](\x+0.5, 6)circle(1.3mm);
}

\foreach \x in {2, 3, 5, 6}{
\shade[ball color=red](\x+0.5, 7)circle(1.3mm);
}
\end{tikzpicture}
		\end{center}
\caption{{The DTASEP transformed in the BLIP setting. Symbols $\otimes$ denote Bernoulli weights 1 to the north-east corner of their square. The coloured balls on each horizontal level is the realization of particles that are still in the $7\times7$ grid. At $t= 7$ there are 4 particles in the square. From this and equations \eqref{eq:bliptot}, \eqref{eq:tincr} we have that $G^{(0)}_{7, 7} = 7 - 4 = 3$.}}
\label{fig:pic8}
\end{figure}
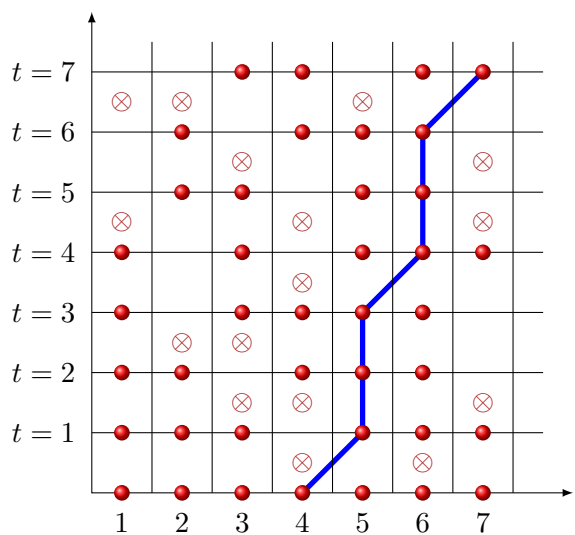

After the spatial locations in the DTASEP at time $\ell=1$ are determined, 
the particles in the graphical construction are at positions $(\eta_{i}(1), 1)$. 
We iterate this procedure for all times $\ell \in \N$. 
 
Then, the environments between graphical DTASEP and BLIP may be coupled via
\[
 1 - \om_{k,\ell} = b_{k + \ell, \ell}.
\]
In \cite{Geo-2010} the following combinatorial identity was proved:
 \be \label{eq:bliptot}
 G^{(0)}_{m, n} = m - \max\{ k: (m - n)\vee1 \le k \le m,\,\, \tau_{k + n - m, k} \le n \}.
\ee
Set 
$ k^* = \max\{ k\le m:  k\geq (m - n)\vee1 ,\,\, \tau_{k + n - m, k} \le n \}\vee 0. 
$
Then 
\be \label{eq:tincr}
\{ G^{(0)}_{m.n} \le m -N \} = \{ N \le k^*\} = \{ \tau_{N + n - m, N} \le n\},
\ee
where the last equality comes form the fact that  $\tau_{N + n - m, N}$  is an increasing random variable in $N$ For a clear pictorial explanation about the coupling, look at Figure \ref{fig:pic8}.
Finally compute 
\begin{align*}
\mathbb P\{G^{(0)}_{m,n} \le m - N\} &= \mathbb P\left\{ N \le \max\{ k: (m - n)\vee1 \le k \le m,\,\, \tau_{k + n - m, k} \le n \} \right\}, \quad \text{ by \eqref{eq:bliptot} }\\
&= \mathbb P\{ \tau_{N + n - m, N} \le n  \}, \quad \text{by \eqref{eq:tincr}} \\
&= \mathbb P\{ \cgs{N + n - m}{N} \le n + N -1\}, \quad \text{by \eqref{eq:bliptoG}}. 
\end{align*}

\subsection{Proof of Theorem \ref{thm:tightness}}

	Recall that $m_n=n/p-xn^a$ and $a\in (0,\nicefrac12]$. Our goal is to prove that the sequence of random variables $n-\bsz{n}{m_n}$ is tight. 
	The main ingredient in the proof is  identity \eqref{eq:disteq}. Set $N=\frac {nq}p - x n^a +k$. Then
	\begin{align}
		\label{eq:55}
		n-m_n+N&=n-\frac np+x n^a +\frac {nq}p - x n^a +k=k.
		\intertext{Since $N(n)$ is eventually monotone, we can invert the expression above and find $n$ in terms of $N$ for sufficiently large $n$ (and hence $N$). In particular,}
		n=n(N)& = \frac pq N + x N^a \rb{\frac pq}^{a+1} + O(N^{2a-1}).
	\end{align}
		To see this we compute
	\begin{align*}
		N(n(N)) & = \frac qp n(N) - x n(N)^a + k \\
		&= \frac qp \rb{\frac pq N + x N^a \rb{\frac pq}^{a+1} 
		+ O(N^{2a-1})} 
		- x \rb{\frac pq N + x N^a \rb{\frac pq}^{a+1} + O(N^{2a-1})}^a+k\\
		&= N+ \rb{\frac pq}^{a} xN^a - x \rb{\frac pq N}^a \rb{1 + x  N^{a-1} \rb{\frac pq}^{a} + O(N^{2a-2})}^a + O(1)\\
 		&= N+ \rb{\frac pq}^{a} xN^a
  		 - x \rb{\frac pq N}^a \rb{1 +  a x  N^{a-1} \rb{\frac pq}^{a} + O(N^{2a-2})} + O(1)\\
		 & =N+O(1).
	\end{align*}
	Therefore,  $n+N-1=\frac Nq+x  \rb{\frac pq}^{a+1} N^a + O(N^{2a-1})$. 	
	Combining \eqref{eq:disteq} and \eqref{eq:55}
	\begin{align} \allowdisplaybreaks
		\P\{ k \le n - G^{(0)}_{m_n,n}\} &= \P\{ G^{(0)}_{m_n,n} \le m_n - N \} \notag \\
		&=\P\Big\{ T_{k, N} \le \frac Nq + x  \rb{\frac pq}^{a+1} N^a + O(N^{2a-1}) \Big\}\notag \\
		&\le\P\Big\{ \max_{j:1 \le j\le k} \sum_{i = 1}^N \zeta_{i,j} \le \frac Nq + x  \rb{\frac pq}^{a+1} N^a + O(N^{2a-1}) \Big\} \notag\\
		&=\P\Big\{\sum_{i = 1}^N \zeta_{i,1} - N \E(\zeta_{11}) \le x  \rb{\frac pq}^{a+1} N^a + O(N^{2a-1}) \Big\}^k \label{eq:CLTleft}.
	\end{align}
	The results follow by first dividing by $\frac{\sqrt{p}}{q}\sqrt{N}$ and the central limit theorem, when we let $n$ (hence $N$) tend to infinity. When $a < 1/2$ the right hand side after scaling tends to 0 and the probability converges to $1/2$. When $a = 1/2$ the right-hand side in the probability converges to $x p q^{-\nicefrac12}$ and the probability to $\Phi(x p q^{-\nicefrac12})$.

	\subsection{Proof of Theorem~\ref{cor:ORhalf}}

We first show the result when $a < 1/2$. 
Using equations \eqref{eq:RnGbound} from Remark \ref{rmk:RnGbound} and \eqref{eq:CLTleft} from the proof of Theorem \ref{thm:tightness},  we have 
\begin{align}
\bP\{ k &< \indreg{\frac{n}{p} -xn^a}n \}\le \bP\{ k < n - G^{(0)}_{\frac{n}{p} -xn^a, n} \} \notag \\
&\phantom{xxxxxxx}=  \left(\bP\Big\{ \frac{\sum_{i = 1}^{N} \zeta_{i,1} - \E(\zeta_{i,1}) N}{\sqrt{\textrm{Var}(\zeta_{i,1})N}} < C_1N^{a-1/2} \Big\}\right)^k \label{eq:aless}
\end{align}
for $C_1$ large enough. As in the proof of Theorem \ref{thm:tightness} we have $N=\frac {nq}p - x n^a +k$ and let $\Phi$ denote the cumulative distribution function of the standard normal distribution. Fix a tolerance $\delta > 0$ satisfying  $\Phi(\delta) + \delta < 1$ and let $n_1(\delta)$ large enough so that  $C_1N^{a - 1/2} < \delta$ 
for all $n > n_1(\delta)$.  Applying the Berry-Esseen theorem to the last line of the last display,
\be \label{eq:sum:bound}
\bP\{ k \le \indreg{\frac{n}{p} -xn^a}n \} \le \Big(\Phi(\delta) + \frac{C}{\sqrt n}\Big)^k \le  \rb{\Phi(\delta) + \delta}^k, \quad \text{for all $n > n_2(\delta)$.}
\ee 
For $n\geq n_0(\delta)=n_1(\delta)\vee n_2(\delta)$ the right hand side of \eqref{eq:sum:bound} is uniformly summable in $k$. Moreover, by \eqref{eq:sum:bound} and the reverse Fatou's Lemma we compute 
\begin{align*}
\varlimsup_{n \to \infty} \E\ab{\indreg{\frac{n}{p} -xn^a}n}  &= \varlimsup_{n \to \infty} \sum_{k =0}^{\infty} \bP\{ k \le \indreg{\frac{n}{p} -xn^a}n \} \\
&\le   \sum_{k =0}^{\infty} \varlimsup_{n \to \infty} \bP\{ k < n - \bsz {\frac{n}{p} -xn^a}{n} \}  \le \sum_{k =0}^{\infty}  2^{-k} = 2, 
\end{align*}
where the penultimate inequality follows from \eqref{eq:RnGbound} and the last from Theorem \ref{thm:tightness}.

The case $a = \nicefrac12$ is slightly more delicate, but the ideas are exactly the same. As before, 
\be\label{eq:ais}
\bP\{ k < \indreg{\frac{n}{p} -x\sqrt{n}}n \}\le  \left(\bP\Big\{ \frac{\sum_{i = 1}^{N} \zeta_{i} - \E(\zeta_1) N}{\sqrt{\textrm{Var}(\zeta_1)N}} < x\frac{p}{\sqrt{q}} +C_0N^{-1/2} \Big\}\right)^k.
\ee
The right-hand side converges to $(\Phi(xpq^{-\nicefrac12}))^k$ and with the same arguments as before,
\[
\varlimsup_{n \to \infty} \E\ab{\indreg{\frac{n}{p} -xn^a}{n}} \le \frac1{1- \Phi(xpq^{-\nicefrac12})}. 
\]

\subsection{Proof of Theorem \ref{thm:iid:regions:ub} }
	When $a \le 1/2$ the result follows from equations \eqref{eq:aless}, \eqref{eq:ais}. For $a \in (\nicefrac12, \nicefrac34]$
	\begin{align*}
		\varlimsup_{n\to \infty}\P\Big\{ \Big(\frac{(px)^2}{4(1-p)} &+ \e\Big)n^{2a-1} \le \indreg{p^{-1}n - xn^a}n  \Big\} \\ 
			&\le \varlimsup_{n\to \infty}\P\Big\{ \Big(\frac{(px)^2}{4(1-p)} + \e\Big)n^{2a-1} \le n - G^{(0)}_{p^{-1}n - xn^a, n}\Big\}=0.
	\end{align*}
	The last inequality follows from \eqref{eq:RnGbound} and the last equality is from \eqref{eq:geoagain}. This gives the second part of the statement.
	
		When $a \in\rb{\nicefrac34,1}$ we can obtain a sharper 
		bound using Lemma \ref{lem:future}.
		
		From the proof of Lemma \ref{lem:naivebound} and Lemma \ref{lem:pis}
		we can find a constant $C_1$ 
		such that  
		$n - G^{(\beta_R)}_{p^{-1}n - xn^a, n} = n -z_R < C_1 n^a$ 
		in probability, as $n$ grows. Therefore, with probability tending to 1 as $n$ grows,
		\be \label{eq:11}
			z_0 -z_R < C_1n^a.
		\ee
	Moreover, since the the number of vertical steps at $\beta =0$
		cannot exceed $ n - G^{(0)}_{p^{-1}n - xn^a, n}$, \eqref{eq:geoagain} gives that with probability tending to 1
		\be\label{eq:12}
			 y_0 -y_R \le  n - G^{(0)}_{p^{-1}n - xn^a, n} <  C_2 n^{2a-1}. 
		\ee 
		Equations \eqref{eq:11}, \eqref{eq:12} now yield a constant $C$ such that  
		\be
			\lim_{n \to \infty}\P\{ z_0 - z_R + y_0 - y_R < C n^a \} = 1. 
		\ee
		Let $A_n$ the event in the probability above. On $A_n$, 
		$
		\sum_{i = 0}^{R-1}\left\{(z_i - z_{i+1}) + (y_i - y_{i+1})\right\} < Cn^a.
		$
		Now we are in a position to use Lemma \ref{lem:future} and finish the proof.
	
	\subsection{Proof of Theorem \ref{thm:TW} (Edge fluctuations for the independent model)}
	\label{sec:flucs}

	We will once more use \eqref{eq:disteq}.
	 Recall that

		\[ x = \frac{2}{\sqrt p}\left(\frac{q}{p}\right)^a\quad\text{ and }\quad y = s\frac{\sqrt{p}}{q}\left(\frac{p}{q}\right)^{\frac{1+a}{3}},\,\,\, s \in \R.\]
We further define an auxiliary parameter $N$ that will go to $\infty$ when $n$ goes to infinity. 
	\be\label{eq:nmono}
	N = N(n)= \frac{q}{p}n - xn^a - yn^{\frac{2-a}{3}} + c_n,
	\ee
	where $c_n$ is given by 
	\be \label{eq:cn}
	c_n = 
		\begin{cases}
			\left(\frac{q}{p}\right)^{2a-1}n^{2a-1}, &1/2<a <2/3, \\
			\left(\frac{q}{p}\right)^{2a-1}n^{2a-1} - (2a-1) x\left(\frac{q}{p}\right)^{2a-2}n^{3a-2}, &2/3\le a <5/7. \\
		\end{cases}
	\ee 
	Note that with $m_n = \frac{1}{p}n - xn^a - yn^{\frac{2-a}{3}}$ we have the relation 
	\be\label{eq:mminusn}
	m_n - n = N - c_n.
	\ee 
	Our goal now is to change $n$ to $N$ and compute $m_n, n, c_n$ in terms of $N$, similarly to the proof of Theorem \ref{thm:tightness}.
	
	\begin{enumerate}
	\item{\bf Step 1: $m_n - n$ and $c_n$ as a function of $N$:} Start from \eqref{eq:nmono} and raise it to 
	the power $2a - 1$. Then, apply Taylor's theorem to obtain  
	\begin{align*}
		N^{2a-1}
		& = \rb{\frac qp\,n}^{2a-1}\rb{ 1-(2a-1)\frac{px}q\, n^{a-1} +O\rb{n^{-\frac{1+a}3}} } 
		= c_n+O(n^{\frac{5a-4}3}).
	\end{align*}
	Note that the equation above holds, irrespective of the value of $a$, as long as 
	$a < 5/7$; for $a\in [0,\nicefrac57)$ the exponent $\frac{5a-4}3<0$,
	so 
	\[
	c_n = N^{2a-1} + o(1)
	\] 
	follows. Therefore,
	a substitution in \eqref{eq:mminusn} yields
	\be \label{eq:argh}
	m_n-n = N - N^{2a-1}+o(1). 
	\ee	

	\item {\bf Step 2:  $n$ as a function of $N$}:
	We begin by writing $n$ as a function of $N$. 
	 Observe that $N(n)$ in equation \eqref{eq:nmono} is an eventually monotone function. 
	 Therefore, for $N$ large enough, there is a well defined inverse $n = n(N)$ 
	 (so that $N(n(N)) = N$). 
	 We cannot directly use a closed formula for the inverse, so we define 
	 the approximate inverse $\ell(N)$ by 
	\[ 
	\ell(N) = \frac{p}{q}N + \frac{2\sqrt p}{q} N^a + y\left(\frac{q}{p}\right)^{\frac{1+a}{3}} N^{\frac{2-a}{3}}. 
	\]
	To see that $\ell(N)$ plays the role of the inverse $n(N)$, substitute $\ell(N)$ in \eqref{eq:nmono} and estimate using 
	a Taylor expansion the distance
	\be \label{eq:almostinverse}
	 |N-N(\ell(N))| =|N - \frac{q}{p}\ell(N) + x_{p,a}\ell(N)^a + y\ell(N)^{\frac{2-a}{3}}| = O(N^{2a-1}).
	\ee
	This implies that $|n(N) - \ell(N)| = o(N^{\frac{2-a}{3}})$; in fact we will show that =
	\be
	\label{eq:ordernl}
		| n(N) - \ell(N)| < c N^{\beta},
	\ee 
	for any $\beta \in (2a -1, \frac{2-a}{3})$.
	Assume for a contradiction that \eqref{eq:ordernl} does not hold for some $c>0$ and for some $\beta > 2a - 1$. Then \allowdisplaybreaks
	 \begin{align*}
	|N - N(\ell(N))| &=  |N(n(N)) -  N(\ell(N)) | \\
	&=\Big| \frac{q}{p}(n(N) - \ell(N)) - x(n(N)^a - \ell(N)^a) \\
	&\phantom{xxxxxxxxx}- y(n(N)^{\frac{2-a}{3}} - \ell(N)^{\frac{2-a}{3}}) + c_{n(N)}-c_{\ell(N)} ) \Big| \\
	&\ge \frac{q}{p}|n(N) - \ell(N)| -x|n(N)-\ell(N)|^a -|y||n(N)-\ell(N)|^{\frac{2-a}{3}}\\
	&\phantom{xxxxxxxxx}- |c_{n(N)} - c_{\ell(N)}| \\
	&\ge CN^{\beta} \text{  for some $C >0$ and $N$ large enough}. 
	 \end{align*}
	 This contradicts \eqref{eq:almostinverse} since $\beta > 2a-1$. 
	 In particular 
	 we have shown that
	\be \label{eq:ntoN}
	\lim_{N \to \infty} \frac{| n(N) - \ell(N)|}{N^{\frac{2-a}{3}}} = \lim_{N \to \infty} \frac{n(N) -  \frac{p}{q}N - \frac{2\sqrt p}{q} N^a - 
	y\left(\frac{q}{p}\right)^{\frac{1+a}{3}} N^{\frac{2-a}{3}} }{N^{\frac{2-a}{3}}} = 0,
	\ee
	and we may write 
	\be \label{eq:n = N}
	n = \frac{p}{q}N + \frac{2\sqrt p}{q} N^a + y\left(\frac{q}{p}\right)^{\frac{1+a}{3}} N^{\frac{2-a}{3}} + o(N^{\frac{2-a}{3}}) = \ell(N) + o(N^{\frac{2-a}{3}}).
	\ee
	\end{enumerate}
	
	To finish the proof we need to be a bit cautious with the integers parts. Define  $k_N$ to be 
	\[
	k_N = \fl{m_n} - n - \fl{N} + \fl{\fl{N}^{2a-1}}.
 	\]
	It follows from \eqref{eq:argh} that $k_N$ is bounded in $N$ (and $n$). Also set $N = \fl{N} + \e_N$.
	Substituting these in equation \eqref{eq:disteq} we compute 
	\allowdisplaybreaks
	\begin{align}
	\P\{ \bsz{\fl{m_n}}{n} &\le n - \fl{\fl{N}^{2a-1}}\} = \P\{  \cgs{\fl{\fl{N}^{2a-1}}}{\fl{N} + k_N} \le n +\fl{ N} -1 \}\notag \\
	&= \P\{   \cgs{\fl{\fl{N}^{2a-1}}}{\fl{N} + k_N} \le \ell(N) + N -1 + n -\ell(N) + \e_N\} \notag \\
	&= \P\Big\{   \cgs{\fl{\fl{N}^{2a-1}}}{\fl{N} + k_N} - \frac{1}{q}N - \frac{\sqrt p}{q} N^a \le y\left(\frac{q}{p}\right)^{\frac{1+a}{3}} N^{\frac{2-a}{3}}  -1 + n -\ell(N) + \e_N\Big\} \notag \\
	&= \P\Bigg\{ \frac{  \cgs{\fl{\fl{N}^{2a-1}}}{\fl{N} + k_N} - \frac{1}{q}\fl{N} - \frac{\sqrt p}{q} \fl{N}^a}{\frac{\sqrt{p}}{q}\fl{N}^{\frac{2-a}{3}}} \le s + o(1) \Bigg\}.\label{eq:argh2}
	\end{align}
The passage time in the probability above can be compared with $\cgs{\fl{N^{2a-1}}}{\fl{N}}$ and satisfies 
\[ \abs{ \cgs{\fl{\fl{N}^{2a-1}}}{\fl{N}} - \cgs{\fl{\fl{N}^{2a-1}}}{\fl{N} + k_N}} <  \sum_{i=0}^{\fl{\fl{N}^{2a-1}}} \sum_{j=-k_N}^{k_N}  \zeta_{i,\fl{N}+ j}.\]
Since $a < \nicefrac57$, the number of geometric random variables in the right-hand side of the inequality is of lower order than $N^{\frac{2-a}{3}}$ and when scaled by it, the double sum vanishes $\P$-a.s. This allows us to remove $k_N$ from \eqref{eq:argh2} and equation \eqref{eq:B-M-TW} now gives the result by taking $n\to\infty$.

	\section{Optimality regions in the alignment model}
	\label{sec:ORalignment}
	
	In this section we prove our results about the alignment model.
	Because of Lemma \ref{lem:cones} and \eqref{eq:shift} 
	it is enough to consider the case where $\alpha=0$. 
	
	 Now it is straight-forward to prove theorems \ref{thm:LPalignment} and \ref{thm:ORprobregions}.
 
 \subsection{Proof of Theorem \ref{thm:LPalignment}} 
 
 Restrict to the full measure set of environments so that Lemma \ref{lem:pis} is in effect. Fix one such environment and assume $n$ is large enough so that statements (1)-(3) of Lemma \ref{lem:pis} hold. Let 
 \[
 g^{(a)}(n) = 
 \begin{cases}
 \sqrt{n \log n}, \quad a \le 1/2,\\
 n^a, \quad a > 1/2. 
 \end{cases}
\]
 Path $\pi_S$ is admissible under any penalty $\beta$, therefore by re-arranging the terms in the inequality of Lemma \ref{lem:pis},
 \begin{align*}
 u_n(a) \cdot e_2 - \sum_{k=1}^{u_n(a)\cdot e_2} \om_{k,k}  - \beta  xn^a \ge n ( 1 + \beta - \beta \abs\abc ) -  \asz{\fl{n\abs\abc - xn^a}}n\beta. 
 \end{align*}
 Now divide both sides by $g^{(a)}(n)$ and take the $\varlimsup$ as $n\to \infty$ to obtain 
 \be
 \varlimsup_{n\to \infty} \frac{n ( 1 + \beta - \beta \abs\abc )-  \asz{\fl{n\abs\abc - xn^a}}n\beta}{g^{(a)}(n)} \le 
 \begin{cases} 
 \sqrt{c_1} - \frac{1}{\abs\abc}, & a \le 1/2,\\
 \frac{1}{\abs\abc(\abs\abc-1)} - \beta x, & a > 1/2.
 \end{cases}
 \ee
 Let $c_1 \searrow \frac{2}{(\abs\abc -1)^2}$ to obtain the upper bound in the theorem. 
 
 For the lower bound, recall that the maximum possible positive weight for $\asz{\fl{n\abs\abc - xn^a}}n\beta$ is $n$ and the smallest possible gap 
 penalty is $\beta (\fl{n\abs\abc - xn^a} - n)$. Therefore 
 \[
  \varliminf_{n\to \infty} \frac{n ( 1 + \beta - \beta \abs\abc )-  \asz{\fl{n\abs\abc - xn^a}}n\beta}{g^{(a)}(n)} \ge 
  \begin{cases}
  0, & a \le 1/2,\\
  - \beta x, & a > 1/2.
  \end{cases}
 \]
 This completes the proof. \qed
 
 \subsection{Proof of Theorem \ref{thm:ORprobregions}}
	 From the previous theorem, we have that for $\beta = 0$, for $\P$-a.e.$~\om$ and any $\e>0$, we can find an $N=N(\omega, \e)$ so that for all $n > N$ 
	 \[
	 n \ge  \asz{\fl{n\abs\abc - xn^a}}n0 \ge n - (C(x,\abs\abc)+\e) g^{(a)}(n).
	 \] 
	 From this equation we immediately obtain that 
	 \be\label{eq:0pen} 
	 z_0 \le n, \quad  y_0 \le 2(C(x,\abs\abc)+ \e)g^{(a)}(n) + \fl{\abs\abc n - xn^{a}} - n.
	 \ee
	 We briefly explain the upper bound for $y_0$. First, any maximal path will always take the minimum number of gaps, which is 
	 $\fl{\abs\abc n - xn^{a}} - n$. 
	 After that, it has to take the correct number of diagonal steps to gain weight equal to  $\asz{\fl{n\abs\abc - xn^a}}n0$. 
	 Now all the remaining steps can either be gaps or mismatches, so we obtain an upper bound if we assume the number of mismatches is zero. 
	 The bound then follows from \eqref{eq:path:steps}.
	 
	 Similarly, for $\beta= \beta_R$, since $\pi_S$ can be optimal under this penalty, Lemma \ref{lem:pis} implies 
	 \be\label{eq:Rpen}
	 z_R \ge n - u_n(a) \cdot e_2 + \sum_{k=1}^{u_n(a)\cdot e_2} \om_{k,k} \ge n - u_n(a) \cdot e_2, \quad\text{and }  y_R = \fl{ \abs\abc n - xn^{a}} -n. 
	 \ee
	 Combine equations \eqref{eq:0pen} and \eqref{eq:Rpen} to obtain for some uniform constant $C$
	 \[
	 z_0 -z_R + y_0 - y_R \le u_n(a) \cdot e_2 + 2(C(x, \abs\abc)+ \e)g^{(a)}(n) \le C g^{(a)}(n),
	 \] 
	 and the result follows from Lemma \ref{lem:future}. \qed

\subsection{Proof of Theorem \ref{thm:ORexpregions}}
 
Lemma \ref{lem:pis}-(3) implies that if $\bar \pi_S$ exits from the north boundary, 
 \begin{align}
	 \label{eq:NorthBdEx}
z_{\beta}(\bar \pi_S) \ge n -  u_n(a) \cdot e_2 + \sum_{k=1}^{u_n(a)\cdot e_2} \om_{k,k}\ge n -  u_n(a) \cdot e_2, \text{ for all $\beta > 0$}.
 \end{align}
Let $B_n$ denote the event \eqref{eq:NorthBdEx} and $D_n$ the event that $\bar \pi_S$ exits from the north boundary. Choose $ c_1 = c_2  = 12/(\abs\abc -1)^2$ in the definition of $u_n(a)$ in \eqref{eq:DefUn}. Then it follows from \eqref{eq:grr} and \eqref{eq:grr2}, using Lemma \ref{lem:MD}, that
$\bar \pi_S$ exits from the north boundary with probability at least $1 - c_0 (n^3 \log n )^{-1}$.
Now, since $z_0 \le n$,
\be\label{eq:b1} 
D_n\subseteq B_n \subseteq \{ z_0 - z_R \le u_n(a) \}. \ee
On the other hand, since $z_0 \ge n -  u_n(a) \cdot e_2$, equation  \eqref{eq:path:steps} implies that 
\[ y_0 \le 2 u_n(a)\cdot e_2 + \fl{\abs\abc n - xn^{a}} - n =  2 u_n(a)\cdot e_2 - y_R.\] 
Therefore 
\be \label{eq:b2} 
D_n\subseteq \{ y_0 - y_R \le 2u_n(a)\cdot e_2 \}. \ee

Combine equations \eqref{eq:b1} and \eqref{eq:b2} to deduce
\begin{align}\label{eq:b3}
 D_n&\subseteq \{ y_0 - y_R \le 2u_n(a)\cdot e_2 \}\cap\{ z_0 - z_R \le u_n(a)\cdot e_2\}
 \subseteq \{ z_0 - z_R + y_0 - y_R \le 3 u_n(a)\cdot e_2\}. \notag
\end{align}
Finally, use \eqref{eq:epiphany} to obtain that for all $n > N =N(a, x)$, 
\be\label{eq:b4}
D_n\subseteq \{ R_{m,n} \le  C (u_n(a)\cdot e_2)^{2/3} \}.
\ee
On the complement of $D_n$ we bound $R$ by $n$, by virtue of \eqref{eq:RnGbound}.
Then for $n$ large enough, 
\begin{align*}
\E(\alreg{\fl{n\abs{\abc} - xn^a}}{n}) &\le  \E(\alreg{\fl{n\abs{\abc} - xn^a}}{n}1\!\!1\{D_n\}) + n \P\{ D^c_n \} \\
&\le  \E(\alreg{\fl{n\abs{\abc} - xn^a}}{n}1\!\!1\{\alreg{\fl{n\abs{\abc} - xn^a}}{n} \le C (u_n(a)\cdot e_2)^{2/3}\}) + n \P\{ D^c_n \} \\
&\le \begin{cases} 
C(x, \abs\abc) ( n \log n )^{1/3} &\quad  a \le 1/2,\\
C(x, \abs\abc) n^{2a/3}, &\quad  a > 1/2.
\end{cases}
\end{align*}
This gives the result. \qed

\bibliographystyle{acm}

\end{document}